\documentclass[11pt]{article}

\usepackage{latexsym, amssymb, amsfonts, amsmath, amsthm, graphics, graphicx, subfigure, bbm, stmaryrd}
\usepackage{epstopdf}

\usepackage[pdftex,colorlinks,linkcolor=blue,menucolor=red,bookmarks=true]{hyperref}
\newcommand{\field}[1]{\mathbb{#1}}
\newcommand{\R}{\field{R}}
\newcommand{\E}{\field{E}}

\newcommand{\p}{\field{P}}
\newcommand{\sign}{\mathrm{sign}}

\newcommand{\var}{\mathrm{Var}}
\newtheorem{theorem}{Theorem}

\newtheorem{proposition}{Proposition}
\newtheorem{definition}{Definition}

\newtheorem{corollary}{Corollary}
\newtheorem{remark}{Remark}

\numberwithin{equation}{section}

\title{\bf{Mass Transportation Proofs of Free Functional Inequalities, and Free Poincar\'e Inequalities}}

\author{Michel Ledoux \\ Institut de Math\'ematiques de Toulouse, Universit\'e de Toulouse, F-31062 Toulouse, France 
 \\ and \\ Ionel Popescu\\ Georgia Institute of Technology, 686 Cherry Street, Atlanta GA, 30332, USA \\ IMAR
21, Calea Grivitei Street 010702-Bucharest, Sector 1
Romania }

\date{} 
\begin{document}

 \maketitle

 \abstract{This work is devoted to direct mass transportation proofs of families of functional inequalities in the context of one-dimensional free probability, avoiding random matrix approximation. The inequalities include the free form of the transportation, Log-Sobolev, HWI interpolation and Brunn-Minkowski inequalities for strictly convex potentials. Sharp constants and some extended versions are put forward. The paper also addresses two versions of free Poincar\'e inequalities and their interpretation in terms of spectral properties of Jacobi operators. The last part
establishes the corresponding inequalities for measures on $\R_{+}$ with the reference example of
the Marcenko-Pastur distribution.}

\section{Introduction}

A distinguished role in the world of functional inequalities is played by the logarithmic Sobolev (Log-Sobolev) inequality and the Talagrand or transportation cost inequality. There is an extensive literature dedicated to these inequalities in the classical setting of Euclidean and Riemannian spaces (cf. e.g. \cite {Bakry}, \cite{L2}, \cite {Villani2}, \cite {Wang}).

Given a probability measure $\nu$ on $\R^{d}$,  the transportation cost inequality states that for some
$\rho >0$ and any other probability measure $\mu$ on $\R^{d}$,
\[\tag{$T(\rho)$}
\rho \, W_{2}^{2}(\mu,\nu)\le E(\mu|\nu).
\]
Here  $W_{2}(\mu,\nu)$ is the Wasserstein distance between $\mu $ and $\nu$ of finite second moment
defined by 
\[
W_{2}(\mu,\nu)=\inf_{\pi\in\Pi(\mu,\nu)}\left( \iint |x-y|^{2}\pi(dx,dy) \right) ^{1/2}
\]
with $\Pi(\mu,\nu)$ denoting the set of probability measures on $\R^{2d}$ with marginals $\mu$ and $\nu$ and 
\[
E(\mu|\nu)=\int \log\frac{d\mu}{d\nu}\, d\mu
\] 
is the relative entropy of $\mu$ with respect to $\nu$ if $\mu<<\nu$ and $+\infty$ otherwise.  The Log-Sobolev inequality is that for any $\mu$
\[\tag{$LSI(\rho)$}
E(\mu|\nu)\le \frac{1}{2\rho} \, I(\mu|\nu)
\]
where
\[
I(\mu|\nu)=\int\Big|\nabla \log \frac{d\mu}{d\nu} \Big|^{2}d\mu
\]
is the Fisher information of $\mu $ with respect to $\nu$ which is defined in the case $\mu<<\nu$ with $\frac{d\mu}{d\nu}$ being differentiable.
A more subtle inequality is the HWI inequality relating entropy ( notice that $E(\mu|\nu)$ is $H(\mu|\nu)$ in \cite{OV} which explains the H), Wasserstein distance W,  and Fisher information I
\[\tag{$HWI(\rho)$}
E(\mu|\nu)\le \sqrt{I(\mu|\nu)} \, W_{2}(\mu,\nu)-\frac{\rho}{2} \,  W_{2}^{2}(\mu,\nu).
\]
Poincar\'e's inequality in this classical context is that for any compactly supported and smooth function $\psi$
on $ \R^d$,  
\[\tag{$P(\rho)$} 
\rho \, \var_{\mu}(\psi)\le \int |\nabla \psi|^{2}\mu (dx)
\]
where $ \var_{\mu}(\psi) = \int \psi^2(x) \mu (dx) - \big ( \int \psi (x) \mu (dx) )\big )^2$ is the variance
of $\psi $ with respect to $\mu$.

Starting with Gaussian measures (\cite{Gross}, \cite{Tal}), these inequalities were established
for measures on $\R^d$ with strictly convex potentials by the Bakry-\'Emery criterion (\cite{Bakry}, \cite{L2},
\cite {Villani2}, \cite {Wang}).
More precisely, if $\nu(dx)=e^{-V(x)}dx$, with $V(x)-\rho |x|^{2}$ convex on $\R^{d}$ for some $\rho>0$, both $T(\rho)$ and  $LSI(\rho)$  hold true.
Otto and Villani generated interest in this 
topic through their remarkable paper \cite{OV}, in which they showed that the logarithmic 
Sobolev inequality implies the trasportation inequality, in a rather general setting.  This connection was actually put further through the stronger $HWI(\rho)$ inequality,
which was shown in \cite{OV} to be valid in the case $V(x)-\rho|x|^{2}$ is convex for some $\rho\in\R$,
When $\rho > 0$,  $LSI(\rho)$ is a consequence of $HWI(\rho)$. 
Subsequently the main result from \cite{OV} was simplified and extended, for example \cite{BL2} and recently \cite{Gozlan} to mention only two sources. Another interesting connection in these families
of functional inequalities is that any of $T(\rho)$, $LSI(\rho)$ or $HWI(\rho)$ imply the Poincar\'e
inequality $P(\rho)$.  

The work \cite {OV} by Otto and Villani input in a powerful way the use of mass transportation ideas in the context of functional inequalities.  Starting from this, Cordero-Erausquin used in \cite{Eras} direct convexity arguments combined with mass transport methods to reprove the Log-Sobolev, transportation and HWI inequalities for measures with strictly convex potentials. The strategy is going back to the original approach of \cite{Tal} to the transportation inequality
(see also \cite{Blower}).

In the world of free probability, as it was shown by Ben Arous and Guionnet in \cite{BAG1}, one can realize the free entropy as the rate function of the large deviations for the distribution of eigenvalues of some $n\times n$ complex random matrix ensembles (see also \cite{J}).   To wit a little bit here, let $V:\R\to\R$ be a nice function with enough growth at infinity and define the probability distribution 
\[
\p_{n}(dM)=\frac{1}{Z_{n}} \, e^{-n\mathrm{Tr}_{n}(V(M))}dM
\]
on the set $\mathcal{H}_{n}$ of  complex Hermitian $n\times n$ matrices where $dM$ is the Lebesgue measure on $\mathcal{H}_{n}$.  For a matrix $M$, let
$\mu_{n}(M)= \frac{1}{n}\sum_{k=1}^{n}\delta_{\lambda_{k}(M)}$ be the distribution of eigenvalues of $M$.  These are random variables with values in $\mathcal{P}(\R)$, the set of probability measures on $\R$ which converge almost surely to a non-random measure $\mu_{V}$ on $\R$.   For a measure $\mu$ on $\R$, its the logarithmic energy with external field $V$ is defined by
\[
E(\mu)=\int V(x)\mu(dx)-\iint \log|x-y| \, \mu(dx)\mu(dy).
\]
The minimizer of $E(\mu)$ over all probability measures on $\R$ is exactly the measure $\mu_{V}$.  From \cite{BAG1} we learned that the distributions of $\{\mu_{n}\}_{n\ge1}$ under $\p_{n}$ satisfy a large deviations principle with scaling $n^{2}$ and rate function given by 
\[
R(\mu)=E(\mu)-E(\mu_{V})
\]
The example of the quadratic
potential $V(x) = x^2$ defining the paradigmatic Gaussian Unitary Ensemble in random matrix theory gives rise
to the celebrated semicircular law as equilibrium measure.

Within this random matrix framework, if $V(x)-\rho x^{2}$ is smooth and convex for some $\rho>0$, then the function $\phi(M)=\mathrm{Tr}_{n}(V(M))$ is strongly convex ($\Phi(M)-n\rho |M|^{2}$ is convex) on $\R^{n^{2}}=\mathcal{H}_{n}$. An application of the classical $LSI(n\rho)$ on $\mathcal{H}_{n}$ for large $n$ was used by Biane \cite{Biane2} to prove a Log-Sobolev inequality in the context of one-dimensional free probability which holds (cf. \cite{HPU1}) in the following form
\begin{equation}\label{e:i1}
 E(\mu)-E(\mu_{V})\le \frac{1}{4\rho} \, I(\mu)
\end{equation}
for any probability measure $\mu$ on $\R$ whose density with respect to the Lebegue measure is in $L^{3}(\R)$, where 
\[
I(\mu)=\int\big (H\mu(x)-V'(x) \big )^{2}\mu(dx)
\]
with $H\mu= 2\int \frac{1}{x-y}\mu(dx)$ being the Hilbert transform of $\mu$.  

More precisely, Biane and Voiculescu used the free Ornstein Uhlenbeck 
process and the complex Burger equation.  Using the large random matrix strategy,  Hiai  Petz and Ueda \cite{HPU1} reproved and extended the result of Biane and Voiculescu in the following form.  If $V(x)-\rho x^{2}$ is convex for some $\rho>0$, then for every probability measure $\mu $ on $\R$,
 \begin{equation}\label{e:i2}
 \rho \, W_{2}^{2}(\mu,\mu_{V})\le E(\mu)-E(\mu_{V}).
 \end{equation}
 
Later, the first author \cite{L} gave a simpler proof of \eqref{e:i1} and \eqref{e:i2} based on a free version of the geometric Brunn-Minkowski inequality obtained as a random matrix limiting case of its classical counterpart. He also showed the free analog of the Otto-Villani theorem indicating that the free Log-Sobolev inequality implies the free transportation inequality \eqref{e:i2}.

The first scope of this paper is to provide direct proofs of the preceding functional inequalities in free probability without random matrix approximation. The second author of this paper in \cite{popescu} gave a simple proof of the transportation inequality \eqref{e:i2} on the same line of ideas as in \cite{Tal} for the classical case where random matrix theory is entirely avoided.   

In this paper, following the approach of Cordero-Erausquin \cite{Eras} (see also \cite {Blower}), we use a combination of mass transport and convex analysis which apply to strictly convex potentials.   The methods allow us besides to enlarge the class of potentials under consideration, in particular in instances which lack a proper random matrix approximation. For example, we cover potentials $V$ on the line such that $V(x)-\rho |x|^{p}$ is convex for some $\rho>0$ and $p>1$ as well as a class of bounded perturbations of convex potentials.   Using this approach, we present here an HWI free inequality for various cases of potentials. For the case $V(x)-\rho x^{2}$ convex for some $\rho\in\R$, this is
\begin{equation}\label{e:i3}
E(\mu)-E(\mu_{V})  \le \sqrt{I(\mu)} \, W_{2}(\mu,\mu_{V})-\rho \, W_{2}^{2}(\mu,\mu_{V}).  
\end{equation}
Also a Brunn-Minkovski inequality receives a direct proof as well.  

One interesting byproduct of our method is that some constants may be shown to be sharp.  For the case of a quadratic $V$, equations \eqref{e:i1},  \eqref{e:i2} and \eqref{e:i3} are sharp.  

Another topic discussed here in Section~\ref{potindep} is a free form of the transportation inequality which does not depend on the potential and that might be thought of as a version of the celebrated Pinsker inequality comparing total variation distance and entropy between probability measures.   As opposed to the classical case, the free counterpart is more delicate.    
  
The second part of this work is devoted to free one-dimensional Poincar\'e inequalities. Using random matrix approximations and the classical Poincar\'e inequality,  we first give an ansatz to what could be a possible Poincar\'e inequality in the free probability world.  In the case of $V(x)-\rho x^{2}$ convex for some $\rho>0$,  such that the measure $\mu_{V}$ has support $[-1,1]$, this states as,
\begin{equation}\label{e:i4}
\int \phi'(x)^{2}\mu_{V}(dx)\ge \frac{\rho}{2\pi^{2}}\int_{-1}^{1}\int_{-1}^{1}\left(\frac{\phi(x)-\phi(y)}{x-y}\right)^{2}\frac{1-xy}{\sqrt{1-x^{2}}\sqrt{1-y^{2}}} \, dxdy,\end{equation}
for any smooth function $\phi$ on the interval $[-1,1]$.    

There is also a second version of the Poincar\'e which is discussed in \cite{Biane2} for the case of the semicircular law. This inequality has a natural meaning in the context of free probability as the derivative $\nabla \phi$ of a function from the classical $P(\rho)$ is replaced by the noncommutative derivative $\frac{\phi(x)-\phi(y)}{x-y}$, and thus our second version takes the form
\begin{equation}\label{e:i5}
\iint \left( \frac{\phi(x)-\phi(y)}{x-y}\right)^{2}\mu(dx)\mu(dy)\ge C\, \var_\mu (\phi)
\quad\text{for every}\: \: \phi\in C_{0}^{1}(\R).
\end{equation}
As opposed to \eqref{e:i4} which requires certain conditions on the measure $\mu_{V}$,  it turns out that \eqref{e:i5} is always satisfied for any compactly supported measure $\mu$ with some constant.  As was shown in \cite{Biane2} for the semicircular law, one can completely characterize the distribution in terms of the constant $C$.      

After the use of convexity, inequality \eqref{e:i4}  may actually be interpreted as a spectral gap as follows.  On $L^{2}\left(\frac{\mathbbm{1}_{[-2,2]}(x)dx}{\sqrt{4-x^{2}}}\right)$ take the Jacobi operator 
\[
Lf=-(1-x^{2})f''(x)+xf'(x)
\]
and the counting number operator defined by
 \[
 NT_{n}=nT_{n}
 \] 
where $T_{n}$ are the Chebyshev polynomials of the first kind, which are orthogonal in $L^{2}\left(\frac{\mathbbm{1}_{[-2,2]}(x)dx}{\sqrt{4-x^{2}}}\right)$.  Then, \eqref{e:i4} for $V(x)=x^{2}/2$ is equivalent to 
\[
L \ge N.
\]

Inequality \eqref{e:i5} in the case of $V(x)=x^{2}/2$ can also be seen as the spectral gap for the counting number operator on $L^{2}\left(\mathbbm{1}_{[-2,2]}(x)\sqrt{4-x^{2}}dx\right)$ with respect to the basis given by the Chebyshev polynomials of second kind.  A more general situation is discussed in Section 9 which includes both versions of the Poincar\'e inequalities.

As we mentioned already, in the classical setting,  the Log-Sobolev and the transportation inequality imply the Poincar\'e inequalities.   We do not have a satisfactory picture of these implications in the free context, for any of the two versions of the Poincar\'e inequality discussed here.

In the final part, we investigate the preceding families of functional inequalities for probability measures supported on the positive real axis. The random matrix context is the one of Wishart ensembles with reference measure the Marcenko-Pastur distribution as opposed to the semicircular law, and the free functional inequalities correspond formally to the case of potentials $V(x)=rx-s\log(x)$ for $r > 0$, $s \geq 0$ on $\R_{+}$. Using the mass transportation method, we prove transportation, Log-Sobolev and HWI inequalities which were not investigated previously. A version of the Poincar\'e inequality is also discussed. 
 
The structure of the paper is as follows.  Sections 2, 4, 5 and 6 deal with the mass transportation proofs of respectively the transportation, Log-Sobolev, HWI and Brunn-Minkowski inequalities. Section 3 studies transportation inequalities which involve some metric on the probabilities and which are independent of the potential $V$. Sections 7 and 8 are devoted to the two versions of the Poincar\'e inequality in the free context, related in Section 9 through Jacobi operators. Section 10 investigates the preceding inequalities with respect to the Marcenko-Pastur distribution and its convex extensions.

\section{Transportation Inequality}
Throughout this paper we consider lower semicontinuous potentials $V:\R\to\R$ such that 
\begin{equation}\label{c}
\lim_{|x|\to\infty}\big (V(x) - 2\log |x| \big )=\infty.
\end{equation}
For a given Borel set $\Gamma\subset \R$,  denote by $\mathcal{P}(\Gamma)$ the set of probability measures supported on $\Gamma$.  

The logarithmic energy with external potential $V$ is defined by
\[
E_{V}(\mu):=\int V(x)\mu(dx)-\iint \log|x-y| \, \mu(dx)\mu(dy).
\]
whenever both integrals exist and have finite values.  In particular for measures $\mu$ which have atoms, $E_{V}(\mu)=+\infty$ because the second integral is $+\infty$.  

It is known (see \cite{ST} or \cite{Deift1}) that under condition \eqref{c} there exists a unique minimizer of $E_{V}$ in the set $\mathcal{P}(\R)$ and the solution $\mu_{V}$ is  compactly supported.  The variational characterization of the minimizer $\mu_{V}$ (cf. \cite[Theorem 1.3]{ST}) is that for a constant $C\in\R$, 
\begin{equation}\label{VC1}
\begin{split}
V(x)&\ge  2\int \log|x-y| \, \mu_{V}(dy)+ C \quad\text{for quasi-every}\:x\in \R   \\ 
V(x)&= 2\int \log|x-y| \, \mu_{V}(dy)+ C  \quad\text{for quasi-every}\: \: x\in \mathrm{supp} (\mu_{V}), 
\end{split}
\end{equation}
where $\mathrm{supp}(\mu_V)$ stands for the support of $\mu$.  If $\mu$ is such that $E_{V}(\mu)<\infty$, then Borel quasi-everywhere sets have $\mu$ measure $0$ and thus the properties above hold almost surely with respect to $\mu$.

For simplicity of the notation, we will drop the subscript $V$ from $E_{V}$ unless the dependence of the potential has to be highlighted.   

Now we summarize some known facts about the equilibrium measure and its support as one can easily deduce them from \cite[Chapter IV]{ST} and \cite[Chapter 6]{Deift1}. 

\begin{theorem}  
\begin{enumerate}
\item Let $V$ be a potential satisfying \eqref{c} and $\alpha\ne0,\beta\in\R$.  Set $V_{\alpha,\beta}(x)=V(\alpha x+\beta)$.  Then, $\mu_{V_{\alpha,\beta}}=((id-\beta)/\alpha)_{\#}\mu_{V}$ and 
\begin{equation}\label{e:141}
E_{V}(\mu_{V})=E_{V_{\alpha,\beta}}(\mu_{V_{\alpha,\beta}})-\log|\alpha|.
\end{equation}
\item If $V$ is convex satisfying \eqref{c}, then the support of the  equilibrium measure $\mu_{V}$ consists of one interval $[a,b]$ where $a$ and $b$ solve the system
\begin{equation}\label{e:0}
\begin{cases}
\frac{1}{2\pi}\int_{a}^{b}V'(x)\sqrt{\frac{x-a}{b-x}}dx=1 \\ 
\frac{1}{2\pi}\int_{a}^{b}V'(x)\sqrt{\frac{b-x}{x-a}}dx=-1.
\end{cases}
\end{equation}
\item Let $V$ be either a  $C^{2}$ satisfying \eqref{c} whose equilibrium measure has support $[a,b]$.  Then the equilibrium measure $\mu_{V}$ has density $g(x)$, given by 
\begin{equation}\label{e:16}
g(x)=\mathbbm{1}_{[a,b]}(x)\frac{\sqrt{(x-a)(b-x)}}{2\pi^{2}}\int_{a}^{b}\frac{V'(y)-V'(x)}{(y-x)\sqrt{(y-a)(b-y)}} \, dy.
\end{equation}
\item If $V$ is $C^{2}$, then 
\begin{equation}\label{vc}
V'(x)=p.v. \int\frac{2}{x-y} \, \mu_{V}(dx)\quad\text{for}\: \mu_{V}-a.s.\:\text{all}\:x\in \mathrm{supp}(\mu_{V}),
\end{equation}
where $p.v.$ stands for the principal value integral.  Notice that the principal value makes sense as $\mu_{V}$ has a continuous density.  

\end{enumerate}
\end{theorem}

We mention as a basic example that if $V(x) = \rho x^2$ is quadratic, then $\mu _V$ is the
semicircular law
\[
\mu_{V}(dx)=\mathbbm{1}_{[-\sqrt {2/\rho },\sqrt {2/\rho }]}(x) \sqrt{2\rho- \rho^{2} x^{2}}  \,  \frac{dx}{2\pi } \, .
\]

In this work, for $p \geq 1$,
we use $W_{p}(\mu,\nu)$ for the Wasserstein distance on the space of probability measures on $\R$ defined as
\begin{equation}\label{e:wass}
W_{p}(\mu,\nu)=\inf_{\pi\in\Pi(\mu,\nu)}\left( \iint |x-y|^{p}\pi(dx,dy) \right) ^{1/p}
\end{equation}
with $\Pi(\mu,\nu)$ denoting the set of probability measures on $\R^{2}$ with marginals $\mu$ and $\nu$.   Note here that if $\theta$ is the (non-decreasing)  transport map such that $\theta_{\#}\mu=\nu$, then 
\begin{equation}\label{e:w}
W_{p}^{p}(\mu,\nu) = \int \big |\theta(x)-x \big |^{p}\nu(dx).
\end{equation}
For a detailed discussion on this topic we refer the reader to \cite{Villani2}.

Our first result concerns the free version of the transportation cost inequality. As discussed in the introduction, the first assertion for strictly convex potentials was initially proved by large matrix approximation in \cite{HPU1}. The strategy of proof is inspired from \cite{Tal}, \cite{Blower} and \cite{Eras} (see \cite{popescu}).

\begin{theorem}[Transportation inequality]\label{thm:1}  
\begin{enumerate}
\item If $V$ is $C^{2}$ and $V(x)-\rho x^{2}$ is convex for some $\rho>0$, then for any probability measure $\mu$ on $\R$,
\begin{equation}\label{t:0}
\rho \, W_{2}^{2}(\mu,\mu_{V})\le E(\mu)-E(\mu_{V}).
\end{equation}
If $V(x)=\rho x^{2}$, then the equality in \eqref{t:0} is attained for measures $\mu=\theta_{\#}\mu_{V}$, with $\theta(x)=x+m$,  therefore the constant $\rho$ in front of $W_{2}^{2}(\mu,\mu_{V})$ is sharp. 
\item Assume that $V$ is $C^{2}$, convex and $V''(x)\ge \rho>0$ for all $|x|\ge r$. Then, there is a constant $C=C(r,\rho,\mu_{V},V)>0$, such that 
\begin{equation}\label{t:2}
C \, W_{2}^{2}(\mu,\mu_{V})\le E(\mu)-E(\mu_{V}).
\end{equation}
\item In the case $V$ is $C^{2}$ and  $V(x)-\rho |x|^{p}$ is convex for some real number $p>1$, then, for any probability measure $\mu$ on $\R$,
\begin{equation}\label{t:1}
c_{p} \rho \, W_{p}^{p}(\mu,\mu_{V})\le E(\mu)-E(\mu_{V})
\end{equation}
where $c_{p}=\inf_{x\in\R}\left( |1+x|^{p}-|x|^{p}-p \sign(x)|x|^{p-1}\right)>0$.  

\end{enumerate}
\end{theorem}

\begin{proof} 
\begin{enumerate}
\item Since there is nothing to prove in the case $E(\mu)=\infty$, we assume that $E(\mu)<\infty$.   In this case we also have that the measure $\mu$ and $\mu_{V}$ both have second finite moments.   

Now we take the non-decreasing transportation  map $\theta$ such that $\theta_{\#}\mu_{V}=\mu$ which exists due to the lack of atoms of $\mu_{V}$. Using the transport map $\theta$, we first write
\begin{align}\label{e:1}
 E(\mu)-E(\mu_{V}) 
&= \int\big ( V(\theta(x))-V(x)-V'(x)(\theta(x)-x)\big )\mu_{V}(dx)\\ 
&\notag\quad+\iint \left(\frac{\theta(x)-\theta(y)}{x-y}-1- \log\frac{\theta(x)-\theta(y)}{x-y}\right)\mu_{V}(dx)\mu_{V}(dy)
\end{align}
where in between we used the variational equation \eqref{vc} to justify that 
\[
\int V'(x) \big (\theta(x)-x \big )\mu_{V}(dx)=2\iint \frac{\theta(x)-x}{x-y} \, \mu_{V}(dy)\mu_{V}(dx)
=\iint \frac{(\theta(x)-x)-(\theta(y)-y)}{x-y}\, \mu_{V}(dy)\mu_{V}(dx).
\]

Since $V(x)-\rho x^{2}$ is convex, for any $x,y$ the following holds
\[
V(y)-V(x)-V'(x)(y-x)\ge \rho \big (y^{2}-x^{2}-2x(y-x)\big )=\rho(y-x)^{2}.
\]
On the other hand since $a-1\ge \log (a)$ for any $a\ge 0$, equations \eqref{e:1} and \eqref{e:w} yield \eqref{t:0}.

In the case $V(x)=\rho x^{2}$  it is easy to see that for $\theta(x)=x+m$, all inequalities involved become equalities, thus we attain equality in \eqref{t:0} for translations of $\mu_{V}$.  

\item We start the proof with \eqref{e:1}, whereas this time we need to exploit the logarithmic term to get our inequality.   The idea is to use the strong convexity where $\psi(x):=\theta(x)-x$ takes large values and for small values of $\psi(x)$ we try to compensate this with the second integral of \eqref{e:1}.

Notice in the first place that by Taylor's theorem we have that 
\begin{equation}
V(y)-V(x)-V'(x)(y-x)=(y-x)^{2}\int_{0}^{1}V''\big ((1-\tau) x+ \tau y \big )(1-\tau)d\tau.
\end{equation}
Now, let us assume that the support of the equilibrium measure $\mu_{V}$ is $[a,b]$.  Next,  $V''(x)\ge0$ and $V''(x)\ge \rho$ for $|x|\ge r$, implies that for $|y|\ge 2r+2\max\{|a|,|b|\}$, we obtain that 
\[
V(y)-V(x)-V'(x)(y-x)\ge (y-x)^{2}\int_{1/2}^{1} V''\big ((1-\tau)x+\tau y\big )(1-\tau)d\tau \ge  \rho(y-x)^{2}/8 \quad\text{ for any}\: \:  x\in[a,b].  
\]
Now write $\theta(x)=x+\psi(x)$.  Thus using \eqref{e:1}, and denoting $R=2r+2\max\{|a|,|b| \}$ we continue with
\begin{align}\label{e:14}
\notag\int \big (V(\theta(x)) -V(x)-V'(x)(\theta(x)-x) \big )\mu_{V}(dx)
 & \ge  \frac{1}{2}\int \psi^{2}(x)\int_{0}^{1}V''\big (x+\tau \psi(x) \big )(1-\tau)d\tau\mu_{V}(dx)  \\
&  \ge \frac{\rho}{16}\int_{|\psi|\ge R}\psi^{2}(x)\mu_{V}(dx).  
\end{align}
This inequality provides a lower bound of the first term in \eqref{e:1}.   Further, it is not hard to check that 
\begin{align}\label{e:15}
\notag\int_{|\psi|\ge R}\psi^{2}(x)\mu_{V}(dx)
& = \frac{1}{2}\int \mathbbm{1}_{|\psi|\ge R}(x)\psi^{2}(x)\mu_{V}(dx)+\frac{1}{2}\int \mathbbm{1}_{|\psi|\ge R}(y)\psi^{2}(y)\mu_{V}(dy) \\ 
& \ge \frac{1}{8} \iint \mathbbm{1}_{|\psi(x)-\psi(y)|\ge 2R}(x,y) 
\big |\psi(x)-\psi(y) \big|^{2}\mu_{V}(dx)\mu_{V}(dy).
\end{align}

Now we treat the second integral on the left hand side of \eqref{e:1}.  Use that
$t-\log(1+t)\ge |t|-\log(1+|t|)$ for any $t>-1$ together with the fact that $t-\log(1+t)$ is an increasing function for $t\ge0$ to argue that 
\begin{align}\label{e:16'}
\notag\iint &\left(\frac{\psi(x)-\psi(y)}{x-y}- \log\left(1+\frac{\psi(x)-\psi(y)}{x-y}\right)\right)\mu_{V}(dx)\mu_{V}(dy) \\ &\ge \iint \left(\frac{|\psi(x)-\psi(y)|}{b-a}- \log\left(1+\frac{|\psi(x)-\psi(y)|}{b-a}\right)\right)\mu_{V}(dx)\mu_{V}(dy)
\end{align}

Further, for $s\ge 0$ and $u,v>0$ we have
\[
us^{2}+s-\log(1+s)\ge 
\begin{cases}
\frac{v-\log(1+v)}{v^{2}}s^{2} & 0\le s\le v\\
us^{2} &v\le s
\end{cases}
\ge \min\left\{u,\frac{v-\log(1+v)}{v^{2}}\right\}s^{2}.
\]
This inequality used for $u=\frac{\rho(b-a)^{2}}{128}$ and $v=\frac{2R}{b-a}$ in combination with \eqref{e:15} and \eqref{e:16'} yields for the choice of  $c=\min\{u,(v-\log(1+v))/v^{2}\}$ that
\begin{equation}\label{e:17}
\begin{split}
&\frac{\rho}{16}\int_{|\psi|\ge R}\psi^{2}(x)\mu_{V}(dx)+\iint \left(\frac{\psi(x)-\psi(y)}{x-y}- \log\left(1+\frac{\psi(x)-\psi(y)}{x-y}\right)\right)\mu_{V}(dx)\mu_{V}(dy) \\ 
&\quad \ge c\iint \big (\psi(x)-\psi(y)\big )^{2}\mu_{V}(dx)\mu_{V}(dy) 
=c\left[\int\psi^{2}(x)\mu_{V}(dx)-\left(\int \psi(x)\mu_{V}(dx) \right)^{2}\right].
\end{split}
\end{equation}
This shows that $E(\mu)-E(\mu_{V})$ is bounded below by a constant times the variance of $\psi$.    Notice that $W_{2}^{2}(\mu,\mu_{V})=\int \psi^{2}(x)\mu_{V}(dx)$ and in order to complete the proof we have to replace the variance of $\psi$ by the integral of $\psi^{2}$ with respect to $\mu_{V}$.  This boils down to estimating the $\mu_{V}$ integral of $\psi$ in terms of the integral of $\psi^{2}$. 

 To this end,  use Cauchy's inequality: 
\begin{align*}
\notag\left(\int \psi(x)\mu_{V}(dx) \right)^{2}
\le 
& \int \psi^{2}(x)\left(1 + \frac{1}{2c}\int_{0}^{1}V''\big (x+\tau\psi(x)\big )(1-\tau)d\tau \right)\mu_{V}(dx)\\ 
& \times\int\frac{1}{1 + \frac{1}{2c}\int_{0}^{1}V''(x+\tau\psi(x))(1-\tau)d\tau }\mu_{V}(dx).
\end{align*}
This inequality combined with equations \eqref{e:1}, \eqref{e:14} and \eqref{e:17}, results with 
\begin{align*}
E(\mu)-E(\mu_{V}) &\ge \int \psi^{2}(x)\left(c + \frac{1}{2}\int_{0}^{1}V''\big (x+\tau\psi(x)\big )(1-\tau)d\tau \right)\mu_{V}(dx)\\ 
& \quad \times\int\frac{\int_{0}^{1}V''(x+\tau\psi(x))(1-\tau)d\tau}{2c + \int_{0}^{1}V''(x+\tau\psi(x))
(1-\tau)d\tau } \, \mu_{V}(dx) \\ 
& \ge c\int\frac{\int_{0}^{1}V''(x+\tau\psi(x))(1-\tau)d\tau}{2c + \int_{0}^{1}V''(x+\tau\psi(x))(1-\tau)d\tau }
   \, \mu_{V}(dx) W_{2}^{2}(\mu,\mu_{V}),
\end{align*}
where here we used the convexity encoded into $V''\ge0$ and the fact that $W_{2}^{2}(\mu,\mu_{V})=\int \psi^{2}(x)\mu_{V}(dx)$ to get the lower bound of the first integral.  

From the previous inequality, it becomes clear that we are done as soon as we prove that the quantity in front of $W_{2}^{2}(\mu,\mu_{V})$ is bounded from below by a positive constant uniformly in $\psi$.   To carry this out,  notice that $V''$ can not be identically zero on $[a,b]$.  Indeed, if $V''$ were identically zero on $[a,b]$, then we would have that $V'(x)=K$ for all $x\in[a,b]$, and this plugged into equation \eqref{e:0}, yields that $K(b-a)=2$ and $K(b-a)=-2$, a system without a solution.  Therefore $V''$ is not identically $0$ on $[a,b]$.   
If $|\psi(x)|>R$, then $V''(x+\tau\psi(x))\ge\rho$ for $1/2\le \tau<1$, which implies $\int_{0}^{1}V''(x+\tau\psi(x))(1-\tau)d\tau\ge\rho/8$.   On the other hand, if $|\psi(x)|\le R$, then 
\[
\int_{0}^{1}V'' \big (x+\tau\psi(x) \big )(1-\tau)d\tau
\ge \int_{0}^{\delta}V'' \big (x+\tau\psi(x)\big )(1-\tau)d\tau
\ge \frac{\delta}{2} \, \inf_{|y-x|\le \delta R}V''(y)
\]
for all $0\le \delta\le 1$. Define
\[
w(x)=\sup_{\delta\in[0,1]}\min\left\{ \frac{\rho}{8} \, ,\frac{\delta}{2}\inf_{|y-x|\le \delta R}V''(y) \right\}.
\]
Since $V''$ is not identically $0$ on $[a,b]$, it follows that $w$ is not identically zero on $[a,b]$.  With this we obtain that 
\[
 \int_{0}^{1}V''\big (x+\tau\psi(x)\big )(1-\tau)d\tau\ge w(x) \ge0,
\]
 and then that 
 \[
 c\int\frac{\int_{0}^{1}V''(x+\tau\psi(x))(1-\tau)d\tau}{2c + \int_{0}^{1}V''(x+\tau\psi(x))(1-\tau)d\tau } \, \mu_{V}(dx) \ge C=\int \frac{cw(x)}{2c+w(x)} \, \mu_{V}(dx)>0
 \]
 which finishes the proof of \eqref{t:2} with this choice of $C$.

\item  
For the inequality \eqref{t:1}, we follow the same route as in the proof of \eqref{t:0}, the only change this time being that  $V(x)-\rho |x|^{p}$ is convex, and thus we obtain
 \begin{equation}\label{e:2}
 V(y)-V(x)-V'(x)(y-x)\ge\rho \big (|y|^{p}-|x|^{p}-p\sign(x)|x|^{p-1}(y-x) \big ).
 \end{equation}
Writing $\theta(x)=x+\psi(x)$,  and using \eqref{e:1} together with $a-1\ge \log (a)$ for $a\ge 0$,
 one arrives at
\[
E(\mu)-E(\mu_{V})\ge  \rho\int \left(|x+\psi(x)|^{p}-|x|^{p}-p\sign(x)|x|^{p-1}\psi(x) \right)\mu_{V}(dx).
\]
Now we use the fact that for all $a,b\in\R$,
\begin{equation}\label{e:30}
|a+b|^{p}-|b|^{p}-p \, \sign(b)|b|^{p-1}a\ge c_{p} |a|^{p},
\end{equation}
which applied to the above inequality in conjunction to \eqref{e:w}, yields  inequality \eqref{t:1}. \qedhere
\end{enumerate}
\end{proof}

\begin{remark} 
\begin{enumerate}
\item The $C^{2}$ regularity of $V$ for \eqref{t:0} can be dropped (see \cite{popescu}) but to simplify the presentation here we decided to consider only this case.  
\item  If $V(x)-\rho|x|^{p}$ is convex, then using inequalities \eqref{t:1}, \eqref{t:2} and Young's inequality we obtain that for any $2\le k \le p$, there exists a constant $c=c(k,p,\rho,\mu_{V},V)$ such that 
\[
c \, W_{k}^{k}(\mu,\mu_{V})\le E(\mu)-E(\mu_{V}). 
\]
\item We want to point out that the inequalities \eqref{t:1} and \eqref{t:2} are somehow complementary to each other.  For example, if we take $V(x)=\rho |x|^{p}$ with $p>1$ and  the measure $\mu=\theta_{\#}\mu_{V}$ for  $\theta(x)=x+m$, then equation \eqref{t:1} takes the form
\begin{equation}\label{e:tmp1}
c_{p}m^{p}\le \int \big (|x+m|^{p}-|x|^{p} \big )\mu_{V}(dx) 
\end{equation}
while equation \eqref{t:2} becomes
\[
Cm^{2}\le \int \big (|x+m|^{p}-|x|^{p} \big )\mu_{V}(dx), 
\]
which, because it is easy to check that $\mu_{V}$ is symmetric, is the same as
\begin{equation}\label{e:tmp2}
Cm^{2}\le \int \big (|x+m|^{p}-|x|^{p}-p \, \sign(x)|x|^{p-1}m \big )  \mu_{V}(dx).
\end{equation}
Notice here that \eqref{e:tmp1} is in the right scale for large $m$ as \eqref{e:tmp2} is in the right scale for $m$ close to $0$, because in this case the integrand is of the size $m^{2}$.   It seems that Talagrand's transportation inequality in this context has two aspects, one is the large $W_{p}(\mu,\mu_{V})$ which is dictated by the potential $V$ for large values and results with equation \eqref{t:1} and the small $W_{2}(\mu,\mu_{V})$ regime which is dictated by the repulsion effect of the logarithm and results with equation \eqref{t:2}. 
\item It is not clear whether inequality \eqref{t:2} still holds for the case of a potential $V$ which is not convex.   Of interest would be the particular case $V(x)=ax^{4}+bx^{2}$ for some $a>0$ and $b<0$. This example actually raises the question of the stability of transportation inequality under bounded perturbations.  
\item Very likely the constant $c_{p}$ in \eqref{t:1} is not sharp.  
\end{enumerate}

\end{remark}

\section{Potential Independent Transportation Inequalities}\label{potindep}

In this section, we investigate some potential independent transportation inequalities. A transportation inequality in the form of \eqref{t:2} can not possibly hold without a quadratic growth at infinity.  Also, the proof of \eqref{t:2} might lead to the conclusion that the logarithmic term plays a more important role.  Therefore the natural question one may ask is whether there is a manifestation of this fact in some sort of transportation type inequality which is independent of the potential involved. The main question reduces to hint some appropriate distance one needs to use to replace the Wasserstein distance in Theorem \ref{thm:1}. We investigate in this section several possibilities, starting with the free version of the classical Pinsker's inequality. 

The Pinsker's inequality classically states that (cf. \cite{Csz} and \cite{Kemp})
\[
2\| \mu-\nu \|_{v}^{2}\le E(\mu|\nu)\quad \text{for any}\: \: \mu,\nu \: \: \text{probability measures on}\: \R,
\]
where $\|\mu-\nu \|_{v}$ is the total variation distance between $\mu$ and $\nu$ and $E(\mu|\nu)$ is the relative entropy between $\mu$ and $\nu$.  This in particular shows that if $\mu_{n}$ convergence to $\mu$ in entropy, then $\mu_{n}$ converges to $\mu$ is a very strong sense.  

The same natural question can be posed in the logarithmic entropy context.  For a given potential $V$, is there an inequality of the form 
\[
C \, \| \mu -\mu_{V} \|_{v}^{2}\le E(\mu) -E(\mu_{V})
\]
for a given constant $C>0$ and any probability distribution $\mu$ on $\R$?  

It turns out that these inequalities do not hold for the logarithmic energy.  In fact, we will show that even a weaker inequality of the form
\begin{equation}\label{e:7'}
C \, |F_{\mu}-F_{\mu_{V}}|_{u}^{2}\le E(\mu)-E(\mu_{V})
\end{equation}
does not hold, where $F_{\mu}$ denotes the cumulative function of a  probability measure $\mu$ on the line.   Even though the uniform distance does not have the same widespread use in probability it appears for example in the Berry-Esseen type estimates for the convergence in the central limit theorem.  This is the reason why we consider this distance as the first next best candidate wherever the total variation fails.   Clearly this metric gives a stronger topology as the topology of weak convergence.  

Will construct a counterexample to \eqref{e:7'} in the case of $V(x)=2x^{2}$, for which the equilibrium measure is  
\[
\mu_{V}(dx)=\mathbbm{1}_{[-1,1]}(x) \, \frac{2\sqrt{1-x^{2}}}{\pi} \, dx,
\]
the semicircular law on $[-1,1]$.  Consider now the sequence 
\[
\mu_{n}(dx)=\mathbbm{1}_{[-1,1]}(x) \, \frac{2\sqrt{1-x^{2}}}{\pi} \, dx+\frac{\sum_{k=2}^{2n-1}(-1)^{k}T_{2k+1}(x)}{4(n^{2}-1)\pi\sqrt{1-x^{2}}} \, dx
\]
 where $T_{k}$ is the $k^{th}$ Chebyshev polynomial of the first kind. With these choices we have that
\begin{equation}\label{e:90}
 E(\mu_{n})-E(\mu_{V})\le \frac{\pi^{2}}{\log(n/3)} \,
 |F_{\mu_{n}}-F_{\mu_{V}}|_{u}^{2}\quad \text{for all}\: \:  n\ge 4.
\end{equation}

Let us point out that $\mu_{n}$ is indeed a probability measure.  This requires a little proof but it's entirely elementary and is left to the reader.  

To prove \eqref{e:7'},  notice that since the support of $\mu_{n}$ is the same as the support of $\mu_{V}$, we have from \eqref{VC1} that 
\begin{equation}\label{e:89}
E(\mu_{n})-E(\mu_{V}) = -\iint \log|x-y|(\mu_{n}-\mu_{V})(dx)(\mu_{n}-\mu_{V})(dy).
\end{equation}
Next remark that $\mu_{n}=\cos_{\#}(f_{n}\lambda)$ and $\mu_{V}=\cos_{\#}(g\lambda)$, where $\lambda$ is the Lebesgue measure on $[0,\pi]$ and
\[
f_{n}(t)=\frac{1-\cos(2t)}{\pi}+ \frac{1}{4\pi (n^{2}-1)}\sum_{k=2}^{2n-1} (-1)^{k}\cos((2k+1)t),
\qquad g(t)=\frac{1-\cos(2t)}{\pi} \, .
\] 
and further
\[
-\iint \log|x-y|(\mu_{n}-\mu_{V})(dx)(\mu_{n}-\mu_{V})(dy)=-\int_{0}^{\pi}\int_{0}^{\pi} \log|\cos t-\cos s|h_{n}(t)h_{n}(s)dtds\quad\text{where}\quad h_{n}=f_{n}-g.
\] 
Now we provide a formula for the logarithmic energy we learnt from \cite{H} and have not seen it elsewhere.  Here is a quick description.  Write first $\cos t=(e^{i t}+e^{-it})/2$ and $\cos s=(e^{i s}+e^{-is})/2$ so $|\cos t-\cos s|=|(e^{it}+e^{-it})/2-(e^{is}+e^{-is})|/2=|1-e^{i(t+s)}||1-e^{i(t-s)}|/2$ and so, for $t\ne s$, and $t$ or $s$ not equal to $\pi$, 
\begin{align*}
\log|\cos t -\cos s|& =-\log 2 + {\rm Re}\big (\log(1-e^{i(t+s)})+\log(1-e^{i(t-s)})\big)
=-\log 2 -\sum_{\ell=1}^{\infty} {\rm Re} \big (e^{i\ell(t+s)}/\ell+e^{i\ell(t-s)}/\ell \big ) \\ 
&=-\log 2-\sum_{\ell=1}^{\infty}\frac{2}{\ell}\cos(\ell t)\cos(\ell s) .
\end{align*}
From this, one gets to
\begin{equation}\label{e:140}
-\int_{0}^{\pi}\int_{0}^{\pi} \log|\cos t-\cos s|h_{n}(t)h_{n}(s)dtds=\sum_{\ell=1}^{\infty}\frac{2}{\ell} \left( \int_{0}^{\pi}\cos(\ell t)h_{n}(t)dt \right)^{2}.
\end{equation}
But now, 
\[
\int_{0}^{\pi}\cos(\ell t)h_{n}(t)dt =\frac{1}{4\pi(n^{2}-1) }\sum_{k=2}^{2n-1}(-1)^{k}\int_{0}^{\pi}\cos(\ell t)\cos((2k+1)t)dt=\begin{cases} \frac{(-1)^{(\ell-1)/2}}{8(n^{2}-1)} & 4\le \ell \le 4n \text{ and odd}\\ 0&\text{otherwise}\end{cases}
\]
and thus
\begin{equation}\label{e:201}
-\int_{0}^{\pi}\int_{0}^{\pi} \log|\cos t-\cos s|h_{n}(t)h_{n}(s)dtds=\sum_{\ell =1}^{\infty}\frac{2}{\ell} \left( \int_{0}^{\pi}\cos(\ell t)h_{n}(t)dt \right)^{2}=\frac{1}{32(n^{2}-1)^{2}}\sum_{\ell=2}^{2n-1}\frac{1}{2\ell+1} \, .
\end{equation}
On the other hand $|F_{\mu_{n}}-F_{\mu_{V}}|_{u}=|F_{f_{n}\lambda}-F_{g\lambda}|_{u}=\sup_{x\in[0,\pi]} \big |\int_{0}^{x}h_{n}(t)dt \big |$ and 
\[
\int_{0}^{x}h_{n}(t)dt=\frac{1}{4\pi (n^{2}-1)}\sum_{\ell=2}^{2n-1}\frac{(-1)^{\ell}\sin((2\ell+1) x)}{2\ell+1} \, ,
\]
from which for $x=\pi/4$, we obtain
\begin{equation}\label{e:202-1}
|F_{\mu_{n}}-F_{\mu_{V}}|_{u}=\sup_{x\in[0,\pi]}\left|\int_{0}^{x}h_{n}(t)dt\right|\ge\frac{1}{4\pi (n^{2}-1)} \sum_{\ell=2}^{2n-1}\frac{1}{2\ell+1} \, .
\end{equation}
Combining \eqref{e:201} and \eqref{e:202-1}  we get 
\begin{equation}\label{e:13}
\frac{\pi^{2}}{2\sum_{\ell=2}^{2n-1}\frac{1}{2\ell+1}} \, |F_{\mu_{n}}-F_{\mu_{V}}|_{u}^{2} \ge -\iint \log|x-y|(\mu_{n}-\mu_{V})(dx)(\mu_{n}-\mu_{V})(dy)
\end{equation}
which together with the fact that $\sum_{\ell=2}^{2n-1}\frac{1}{2\ell+1}\ge \frac{1}{2}\log(n/3)$ for $n\ge4$ and \eqref{e:89}, we finally arrive at \eqref{e:90}.

The example shown above has the property that $E(\mu_{n})-E(\mu_{V})$ converges to $0$ when $n$ goes to infinity, and also that $|F_{\mu_{n}}-F_{\mu_{V}}|_{u}$ converges to zero.  Despite the fact that \eqref{e:7'} does not hold, we will see below in Corollary~\ref{c:1} that if $E(\mu_{n})-E(\mu_{V})$ converges to $0$, then $|F_{\mu_{n}}-F_{\mu_{V}}|_{u}$ always converges to $0$.

We consider now a weak form of \eqref{e:7'}.  To do this we define the distance 
\begin{equation}\label{e:d}
d(\mu,\nu)=\sup_{a,b\in\R} \left| \int e^{-|ax+b|}\mu(dx)-\int e^{-|ax+b|}\nu(dx) \right|.
\end{equation}

With this definition we have the following result.
\begin{theorem}\label{p:1}
For any potential $V$ satisfying \eqref{c}, we have that  for any compactly supported measure $\mu$,
\begin{equation}\label{e:57}
4\pi^{3}d^{2}(\mu,\mu_{V})\le E(\mu)-E(\mu_{V}).
\end{equation}
\end{theorem}

\begin{proof} 
Using equations \eqref{c} and \eqref{VC1}, we get  for any compactly supported measure $\mu$ with $E(\mu)$ finite, 
\[
E(\mu)-E(\mu_{V})\ge -\iint \log|x-y|(\mu-\mu_{V})(dx)(\mu-\mu_{V})(dy).
\]
We will prove that for any measures $\mu$ and $\nu$ with compact support such that $-\iint \log|x-y|\mu(dx)\mu(dy)<\infty$ and $-\iint \log|x-y|\nu(dx)\nu(dy)<\infty$, we have that
\begin{equation}\label{e:6}
4\pi^{3}d^{2}(\mu,\nu)\le -\iint \log|x-y|(\mu-\nu)(dx)(\mu-\nu)(dy),
\end{equation}
which shows that \eqref{e:6} implies \eqref{e:57}. 

Now we use  \cite[equation 6.45]{Deift1} to write
\begin{equation}\label{e:700}
-\iint \log|x-y|(\mu-\mu_{V})(dx)(\mu-\mu_{V})(dy)=\int_{0}^{\infty} \frac{|\hat{\mu}(t)-\hat{\mu}_{V}(t)|^{2}}{t} \, dt
\end{equation}
where the hat stands for the Fourier transform, and  continue with 
\[
\int_{0}^{\infty} \frac{|\hat{\mu}(t)-\hat{\nu}(t)|^{2}}{t} \, dt=\frac{1}{2}\int_{-\infty}^{\infty} \frac{|\hat{\mu}(t)-\hat{\nu}(t)|^{2}}{|t|}dt\ge |a|\int_{-\infty}^{\infty}\frac{|\hat{\mu}(t)-\hat{\nu}(t)|^{2}}{a^{2}+t^{2}} \, dt \ge \frac{a^{2}}{\pi}\left|\int_{-\infty}^{\infty} \frac{\left(\hat{\mu}(t)-\hat{\nu}(t)\right)e^{-ict}}{a^{2}+t^{2}} \, dt\right|^{2}
\]
for any $a,c\in\R$ with $a\ne0$.   Further, using the inversion formula for the Fourier transform, one has
\begin{equation}\label{e:56}
\int_{-\infty}^{\infty} \frac{\left(\hat{\mu}(t)-\hat{\nu}(t)\right)e^{-ict}}{a^{2}+t^{2}} \, dt=2\pi\int \hat{\phi} \, (x)(\mu-\nu)(dx)=\frac{2\pi^{2}}{|a|}\int e^{-|a(x+c)|} (\mu-\nu)(dx)
\end{equation}
because for $\displaystyle\phi(t)=\frac{e^{ict}}{a^{2}+t^{2}}$,
\[
\hat{\phi}(x)=\int\frac{e^{i(x+c)t}}{t^{2}+a^{2}} \, dt=\frac{\pi e^{-|a(x+c)|}}{|a|} \, .
\]
From here, \eqref{e:6} follows immediately.  \qedhere

\end{proof}
\begin{remark}
From equation \eqref{e:700} it seems that the distance one should consider should be the Sobolev norm with exponent $-1/2$.   This is another possible candidate to the role of $d$ played here, however not always finite.   We chose the metric $d$ as it's definition is somehow close to uniform norm of the difference of the Laplace transforms of the measures.  It is also always defined and bounded by $1$, thus resembling the total variation distance.    
\end{remark}

The next result is collecting facts about how strong the topology induced by $d$ is.  

\begin{proposition}\label{p:2}  \begin{enumerate} 
\item $d$ is a distance on $\mathcal{P}(\R)$ and if $d(\mu_{n},\mu)\xrightarrow[n\to\infty]{}0$, then $\mu_{n}\xrightarrow[n\to\infty]{}\mu$ in the  weak topology.   In addition $d(\delta_{a},\delta_{b})=1$ for $a\ne b$, thus the topology induced by $d$ is strictly stronger than the weak convergence topology.   
\item For any two probability measures $\mu$ and $\nu$, 
\begin{equation}\label{e:70b}
d(\mu,\nu)\le 2 \, |F_{\mu}-F_{\nu}|_{u}.
\end{equation}
\item If $V$ satisfies condition \eqref{c}, then 
 $E_{V}(\mu_{n})\xrightarrow[n\to\infty]{}E_{V}(\mu_{V})$ implies $|F_{\mu_{n}}-F_{\mu_{V}}|_{u}\xrightarrow[n\to\infty]{}0$.  
\end{enumerate}
\end{proposition}

\begin{proof} \begin{enumerate} 
\item To prove that $d$ is a distance the only non trivial fact is that for two probability measures $\mu$ and $\nu$,  $d(\mu,\nu)=0$ implies $\mu=\nu$.    Thus from equation \eqref{e:56}, we obtain for $a=1$ that for all $c \in \R$,
\[
\int_{-\infty}^{\infty} \frac{\left(\hat{\mu}(t)-\hat{\nu}(t)\right)e^{-ict}}{1+t^{2}} \, dt=0 .
\]
 Since this holds true for any $c\in\R$, it implies that the Fourier transform of the function $t\to \frac{\hat{\mu}(t)-\hat{\nu}(t)}{1+t^{2}}$ is $0$, which means that the function in discussion must be $0$.  This means that $\hat{\mu}=\hat{\nu}$, or equivalently that $\mu=\nu$.  
 
 Let $\mathcal{L}(\mu,\nu)$ stand for the Levy distance which induces the weak topology on $\mathcal{P}(\R)$.   Let $d(\mu_{n},\mu)\xrightarrow[n\to\infty]{}0$.   Assume now that there exists $\epsilon>0$ and a subsequence such that $\mathcal{L}(\mu_{n_{k}},\mu)\ge \epsilon$.  Otherwise said, the sequence $\mu_{n}$ has a subsequence which is not convergent to $\mu$.  Since, we are dealing with probability measures, there is a subsequence $\mu_{n_{k_{l}}}$ which is vaguely convergent to a measure $\nu$ with total mass less than $1$.  This means that for any continuous function $\phi$ which is vanishing at infinity, we have that  
 \[
 \int \phi d\mu_{n_{k_{l}}}\xrightarrow[l\to\infty]{}\int \phi d\nu.
 \]
We can apply this for functions $\phi(x)=e^{-|ax+b|}$ where $a\ne0$ and infer that 
\[
\int e^{-|ax+b|} \mu_{n_{k_{l}}}(dx)\xrightarrow[l\to\infty]{}\int e^{-|ax+b|} \nu(dx)\quad \text{for all } \: \: a\ne 0,b\in\R.
\]
On the other hand, because $d(\mu_{n_{k_{l}}},\mu)\xrightarrow[l\to\infty]{}0$, these considerations result with 
\[
\int e^{-|ax+b|}\mu(dx)=\int e^{-|ax+b|}\nu(dx)\quad\text{for all}\: \: a\ne0,b\in\R.
\]
Further, using the dominated convergence for $b=0$ and $a\to 0$, we obtain that $\nu$ is a probability measure.  From the discussion at the beginning of this proof, it also follows that $\nu=\mu$ and this in turn results with $\mu_{n_{k_{l}}}$ being weakly convergent to $\mu$, a contradiction.  This proves that the convergence in the metric $d$ implies weak convergence.  

It is obvious that $d(\mu,\nu)\le 1$ for any measures $\mu$ and $\nu$.  For the case of discrete measures, we also have that $1\ge d(\delta_{a},\delta_{b})\ge \int e^{-\alpha|x-a|}\delta_{a}(dx)-\int e^{-\alpha|x-a|}\delta_{b}(dx)$ for any $\alpha>0$, which yields that $1\ge d(\delta_{a},\delta_{b})\ge 1-e^{-\alpha|b-a|}$ for all $\alpha>0$.  Letting $\alpha\to \infty$, we get that $d(\delta_{a},\delta_{b})=1$ for $a\ne b$ which shows that convergence in $d$ is strictly stronger than convergence in the weak topology. 
\item From the fact that for any finite positive measure $\mu$, 
\[
\int_{(0,\infty)} (1-e^{-\alpha y})\mu(dx)=\int_{(0,\infty)} \alpha e^{-\alpha y} \mu((y,\infty))dy,
\] 
we deduce that 
\[
\int e^{-\alpha|x-a|}(\mu-\nu)(dx)=\int_{(0,\infty)}\alpha e^{-\alpha y} \big [F_{\mu}(a-y)-F_{\mu}(a+y)-F_{\nu}(a-y)+F_{\nu}(a+y)\big ] dy
\]
which easily yields \eqref{e:70b}.

\item We actually show that if $\mu_{n}$ and $\mu$ are compactly supported probability measures such that 
\[ 
-\iint \log |x-y| \, \mu(dx)\mu(dy)<\infty,\quad-\iint \log |x-y| \, \mu_{n}(dx)\mu_{n}(dy)<\infty
\]
and
\[
\lim_{n\to \infty}\iint \log |x-y|(\mu_{n}-\mu)(dx)(\mu_{n}-\mu)(dy)=0,
\]
then $|F_{\mu_{n}}-F_{\mu}|_{u}\xrightarrow[n\to\infty]{}0$.  
From \eqref{e:6} and the first part, we obtain that $\mu_{n}$ converges weakly to $\mu$.  In addition, none of the measures $\mu_{n}$ or $\mu$ have atoms.   Thus $F_{\mu_{n}}$ and $F_{\mu}$ are continuous functions which combined with the weak convergence implies that $F_{\mu_{n}}$ converges pointwise to $F_{\mu}$.  Since the functions $F_{\mu_{n}}$ and $F_{\mu}$ are distributions of probability measures, it is an easy matter to check that the convergence is actually uniform. \qedhere
\end{enumerate}
\end{proof}

\begin{remark}  We do not know if the topology of convergence in $d$ is the same as the one defined by the metric $|F_{\mu}-F_{\nu}|_{u}$. 
\end{remark}

This result might leave one wondering if a stronger convergence takes place.  In other words, is it true that $E_{V}(\mu_{n})\xrightarrow[n\to\infty]{}E_{V}(\mu_{V})$ implies $\| \mu_{n}-\mu_{V} \|_{v}\xrightarrow[n\to\infty]{}0$?  
To this end, we can consider $V(x)=\log|\frac{|x|+\sqrt{x^{2}-1}}{2}|$ and notice (see \cite[page 46]{ST}) that $\mu_{V}$ is the arcsine law of $[-1,1]$.  Thus if we consider 
\[
\mu_{V}(dx)=\mathbbm{1}_{[-1,1]}(x)\frac{dx}{\pi \sqrt{1-x^{2}}} \, , \qquad \mu_{n}(dx)=\mathbbm{1}_{[-1,1]}(x)\frac{(1-T_{n}(x))dx}{\pi \sqrt{1-x^{2}}} \, ,
\]
then, using the same argument which led us to \eqref{e:140}, with $h_{n}$ there replaced by $h_{n}(x)=\cos(nx)$ here, one arrives at
$ E(\mu_{n})-E(\mu_{V})=\frac{1}{n}$
while the total variation distance is $\|\mu_{n}-\mu_{V}\|_{v}\ge 1/4$.

\section{Log-Sobolev Inequality}

In this section, we develop similarly the mass transportation method to prove the Log-Sobolev inequality in the free context. Note again that, as discussed in the introduction, the first assertion for strictly convex potentials was initially proved by large matrix approximation in \cite{Biane2}.

Before we state the main result, we define inspired by Voiculescu \cite{Voiculescu2}, the relative free Fisher information as
\begin{equation}\label{e:27}
I(\mu)=\int \big (H\mu(x)-V'(x) \big )^{2}\mu(dx)\quad\text{with}\quad H\mu(x)=p.v. \int\frac{2}{x-y} \, \mu(dy). 
\end{equation}
for measures $\mu$ on $\R$ which have density $p=d\mu/dx$ in  $L^{3}(\R)$.   In this case the principal value integral is a function in $L^{3}$.   Otherwise we let $I(\mu)$ be equal to $+\infty$.

\begin{theorem}[Log-Sobolev]\label{t:3}
\begin{enumerate}
\item If $V$ is $C^{2}$ and $V(x)-\rho x^{2}$ is convex for some $\rho>0$, then for any probability measure $\mu$ on $\R$,
\begin{equation}\label{e:26}
E(\mu)-E(\mu_{V})\le \frac{1}{4\rho} \, I(\mu).
\end{equation}
 
Equality is attained for the case $V(x)=\rho x^{2}$ and $\mu=\theta_{\#}\mu_{V}$, where $\theta(x)=x+m$.  Thus the inequality \eqref{e:26} is sharp for translations of $\mu_{V}$.  
\item If $V$ is $C^{2}$ and $V(x)-\rho |x|^{p}$ is convex for some $\rho>0$ and $p>1$, then
for any probability measure $\mu$ on $\R$,
\begin{equation}\label{e:102}
E(\mu)-E(\mu_{V})\le \frac{k_{p}}{\rho^{q/p}} \, I_{q}(\mu)\quad \text{where}\quad I_{q}(\mu)
=\int \big |H\mu(x)-V'(x) \big |^{q}\mu(dx)
\end{equation}
where here $q$ is the conjugate of $p$ i.e. $1/q+1/p=1$ and the constant $k_{p}=(p c_{p})^{q/p}/q$, with $c_{p}$ from \eqref{t:1}.  
\end{enumerate}
\end{theorem}

\begin{proof}
\begin{enumerate}
\item We will assume that the measure $\mu$ has a smooth compactly supported density as the general case follows via approximation arguments discussed in details in \cite{HPU1}.  Take the (increasing) transport map $\theta$ from $\mu_{V}$ into $\mu$.  We write the inequality \eqref{e:26} in the following equivalent way
\begin{align}
\notag\frac{1}{4\rho}&\int \big (H\mu(\theta(x))-V'(\theta(x)) \big )^{2}\mu_{V}(dx)
+\int \big (V(x)-V(\theta(x))-V'(\theta(x)) \big (x-\theta(x)\big )\big)\mu_{V}(dx) \\ 
& \notag -\int \big(H\mu(\theta(x))-V'(\theta(x))\big)(x-\theta(x))\mu_{V}(dx)  \\ 
& +\int H\mu(\theta(x)) \big (x-\theta(x)\big )\mu_{V}(dx) -\iint \log\frac{x-y}{\theta(x)-\theta(y)} \, \mu_{V}(dx)\mu_{V}(dy)\ge0. \label{e:28}
\end{align}
Notice now that from the convexity of $V(x)-\rho x^{2}$, one obtains that 
\begin{equation}\label{e:63}
V(x)-V(\theta(x))-V'(\theta(x)) \big (x-\theta(x)\big )
\ge\rho \big (x^{2}-\theta(x)^{2}-2\theta(x) (x-\theta(x)) \big )
=\rho \big (x-\theta(x) \big )^{2}.
\end{equation}
 
Now,
\begin{equation}\label{e:64}
\int H\mu(\theta(x)) \big (x-\theta(x) \big )\mu_{V}(dx)
=\int \big (x-\theta(x) \big )\int \frac{2}{\theta(x)-\theta(y)} \, \mu_{V}(dy)\mu_{V}(dx)
=\iint \left(\frac{x-y}{\theta(x)-\theta(y)}-1\right)\mu_{V}(dx)\mu_{V}(dy)
\end{equation}
where one has to interpret the second integral here in the principal value sense, however since $\theta$ is increasing, the last integral is actually taken in the Lebesgue sense.   

Using these, equation \eqref{e:28} may be rewritten as
\begin{align*}
\notag\frac{1}{4\rho}\int  \big[ & (H\mu(\theta(x))  -V'(\theta(x))-2\rho(x-\theta(x))\big]^{2}\mu_{V}(dx) \\ 
&+\iint\left(\frac{x-y}{\theta(x)-\theta(y)}-1- \log\frac{x-y}{\theta(x)-\theta(y)}\right)\mu_{V}(dx)\mu_{V}(dy)\ge0
\end{align*}
which is seen to hold since $u-1-\log (u)\ge0$ for $u\ge0$.

Equality is attained for the case $V(x)=\rho x^{2}$ and  $\theta(x)=x+c$, which corresponds to the translations of the measure $\mu_{V}$.  

\item With the same arguments used in the above proof and the proof of Theorem \ref{thm:1}, we use equations \eqref{e:2} and \eqref{e:30} to argue that 
\begin{align*}
\frac{k_{p}}{\rho^{q/p}}&\int \big | H\mu(x)-V'(x) \big |^{q}\mu(dx)-E(\mu)+E(\mu_{V})\notag \\ 
&\ge\notag \int \Big [ \frac{k_{p}}{\rho^{q/p}} \big | H\mu(\theta(x))-V'(\theta(x)\big |^{q}+ 
 \big (V'(\theta(x))-H\mu(\theta(x))\big ) \big (x-\theta(x) \big )+ c_{p}\rho|x-\theta(x)|^{p}\Big ]\mu_{V}(dx) \\ 
& \quad +\iint\left(\frac{x-y}{\theta(x)-\theta(y)}-1- \log\frac{x-y}{\theta(x)-\theta(y)}\right)\mu_{V}(dx)\mu_{V}(dy) \\
& \ge0
\end{align*}
where we used Young's inequality $a^{q}/q+b^{p}/p\ge ab$ for $a,b\ge0$ and the constant $k_{p}=(p c_{p})^{q/p}/q$.\qedhere
\end{enumerate}

\end{proof}

\begin{remark}
It was proved in \cite{L} that a Log-Sobolev inequality always implies a transportation inequality.
\end{remark}

\section{HWI Inequality}

This section is devoted to the free analog of the HWI inequality of Otto and Villani \cite{OV} in the classical context, connecting thus the (free) entropy, Wasserstein distance and Fisher information.  As we will see, the HWI implies the Log-Sobolev inequality for strictly convex potentials. This free HWI inequality was not considered before,
and in particular it is not clear whether there is a random matrix proof, delicate points involving the Wasserstein distance entering into the proof.

\begin{theorem}[HWI inequality]\label{t:hwi}
\begin{enumerate}
\item Assume that $V$ is $C^{2}$ such that for some $\rho\in\R$, $V(x)-\rho x^{2}$ is convex.  Then, for any measure $\mu\in\mathcal{P}(\R)$,
\begin{equation}\label{hwi}
E(\mu)-E(\mu_{V})\le \sqrt{I(\mu)} \, W_{2}(\mu,\mu_{V})-\rho \,  W_{2}^{2}(\mu,\mu_{V}) .
\end{equation}
In the case $V(x)=\rho x^{2}$, the inequality is sharp.  
\item If $V$ is $C^{2}$ and  $V(x)-\rho|x|^{p}$ is convex for some $\rho\ge0$ 
 and $p>1$, then for the same constant $c_{p}$ appearing in Theorem \ref{thm:1}, we have that 
\begin{equation}\label{hwi2}
E(\mu)-E(\mu_{V})\le I_{q}^{1/q}(\mu) \, W_{p}(\mu,\mu_{V})-\rho c_{p} \, W_{p}^{p}(\mu,\mu_{V}),
\end{equation}
where $1/p+1/q=1$.
\end{enumerate}
\end{theorem}

\begin{proof}
\begin{enumerate}
\item  We employ here the notations used in Theorem \ref{t:3} and we will give a proof of the inequality for the case of a measure $\mu$ with smooth and compactly supported density, the general case follows through careful approximations pointed in \cite{HPU1}.    The inequality to be proved can be restated as $\eqref{e:60}+\eqref{e:61}+\eqref{e:62}\ge0$, where
\begin{align}
\notag \eqref{e:60}&=\bigg(\int \big (H\mu(\theta(x))-V'(\theta(x)) \big )^{2}\mu_{V}(dx) \int \big (\theta(x)-x\big )^{2}\mu_{V}(dx)\bigg)^{1/2} \\ 
& \qquad\label{e:60}-\int \big(H\mu(\theta(x))-V'(\theta(x))\big)(x-\theta(x))\mu_{V}(dx)  \\ 
\label{e:61}\eqref{e:61}&=\int \big [V(x)-V(\theta(x))-V'(\theta(x)) \big (x-\theta(x)\big )-\rho \big (\theta(x)-x\big )^{2}\big] \mu_{V}(dx) \\ 
 \label{e:62}\eqref{e:62}&=\int H\mu(\theta(x))\big (x-\theta(x)\big )\mu_{V}(dx) -\iint \log\frac{x-y}{\theta(x)-\theta(y)} \, \mu_{V}(dx)\mu_{V}(dy). 
\end{align}
A simple application of Cauchy's inequality shows that $\eqref{e:60}\ge0$.  
Using convexity of $V(x)-\rho x^{2}$ we have from equation \eqref{e:63}, that $\eqref{e:61}\ge0$. 
Finally, using \eqref{e:64}, we have that 
\[
\eqref{e:62}= \iint\left(\frac{x-y}{\theta(x)-\theta(y)}-1- \log\frac{x-y}{\theta(x)-\theta(y)}\right)\mu_{V}(dx)\mu_{V}(dy)\ge0,
\]
which finishes the proof of \eqref{hwi}.  For the case $V(x)=\rho x^{2}$, we have equality if $\theta(x)=x+m$.  

\item The inequality we want  to prove is equivalent to the statement that $\eqref{e:70}+\eqref{e:71}+\eqref{e:72}\ge0$, where
\begin{align}
\notag\eqref{e:70}&=\left|\int  \big |H\mu(\theta(x))-V'(\theta(x)) \big |^{q}\mu_{V}(dx)\right|^{1/q}\left| \int (\theta(x)-x)^{p}\mu_{V}(dx)\right|^{1/p} \\ 
\label{e:70}&\qquad-\int \big (H\mu(\theta(x))-V'(\theta(x))\big ) \big (x-\theta(x) \big )\mu_{V}(dx)  \\ 
\label{e:71}\eqref{e:71}&=\int \big [ V(x)-V(\theta(x))-V'(\theta(x)) \big (x-\theta(x) \big )-\rho c_{p} \big |\theta(x)-x \big |^{p}\big ]\mu_{V}(dx) \\ 
 \label{e:72}\eqref{e:72}&=\int H\mu(\theta(x)) \big (x-\theta(x) \big )\mu_{V}(dx) -\iint \log\frac{x-y}{\theta(x)-\theta(y)}\mu_{V}(dx) \, \mu_{V}(dy). 
\end{align}
Now, $\eqref{e:70}$ is non-negative thanks to H\" older's inequality, equation $\eqref{e:71}$, follows from the convexity of $V(x)-\rho |x|^{p}$ and the combination of  \eqref{e:2} and \eqref{e:30}, while equation \eqref{e:72} is the same as \eqref{e:62}. \qedhere
\end{enumerate}
\end{proof}

As pointed out in \cite{OV}, HWI inequalities for $\rho>0$ always implies Log-Sobolev.
We give here the following formal corollary of HWI inequality.    

\begin{corollary}\label{c:2} \begin{enumerate}
\item If $\rho>0$, then inequality \eqref{hwi} implies \eqref{e:26} and \eqref{hwi2} implies \eqref{e:27}.   

\item If $V(x)-\rho x^{2}$ is a convex for some $\rho\in\R$, then Talagrand's free transportation inequality with constant $C>\max\{0,-\rho\}$ implies free Log-Sobolev inequality with constant $K=\max\{\rho,\frac{(C+\rho)^{2}}{32C}\}$.  More precisely,
\[
 \forall \mu\in \mathcal{P} (\R), \: \:
C \, W_{2}^{2}(\mu,\mu_{V})\le E(\mu)-E(\mu_{V})\Longrightarrow
 \forall \mu\in \mathcal{P}(\R), \: \: E(\mu)-E(\mu_{V})\le \frac{1}{4K} \, I(\mu).
\]
\item In particular, if $V$ is convex and $C^{2}$ such that $V''(x)\ge \rho>0$ for $|x|\ge r$, then free Log-Sobolev inequality holds with the constant $C>0$ from \eqref{t:2}.  
\end{enumerate}
\end{corollary}

\begin{proof}  \begin{enumerate}
\item It follows as an application of Young's inequality $a^{p}/p+b^{q}/q\ge ab$ for $a,b\ge0$.  
\item For $\rho>0$, everything is clear.  In the case $\rho\le0$, then, from \eqref{hwi}  and Talagrand's transportation inequality, one has for $\delta>0$, that
\[
E(\mu)-E(\mu_{V})\le \sqrt{I(\mu)} \, W_{2}(\mu,\mu_{V})-\rho \, W_{2}^{2}(\mu,\mu_{V})\le 4\delta I(\mu)+\left( \frac{1}{C\delta}-\frac{\rho}{C}\right) \big (E(\mu)-E(\mu_{V}) \big )
\]
which yields for any $\delta>\frac{1}{C+\rho}$
\[
E(\mu)-E(\mu_{V})\le \frac{4C\delta^{2}}{(C+\rho)\delta-1} \, I(\mu).
\]
Taking minimum over $\delta>\frac{1}{C+\rho}$ gives the conclusion.  
\item 
In the case $V$ is convex, $C^{2}$ and strongly convex for large values, part 2 of Theorem \ref{thm:1} does the rest.  \qedhere
\end{enumerate}
\end{proof}

\section{Brunn-Minkowski Inequality}

The (one-dimensional) free Brunn-Minkowski inequality was put forward in \cite{L} again through random matrix approximation. We provide here a direct mass transportation proof similar to the one of its classical (one-dimensional) counterpart (see e.g. \cite{Gardner}). As discussed in \cite{L}, this inequality may be used to deduce in an easy way both the Log-Sobolev and transportation inequalities.

The main result of this section is the following theorem.

\begin{theorem}
Assume that $V_{1},V_{2},V_{3}$ are some potentials satisfying \eqref{c} such that for some $a\in(0,1)$,
\begin{equation}\label{e:17'}
a V_{1}(x)+(1-a)V_{2}(y)\ge V_{3}(ax+(1-a)y)\quad\text{for all}\quad x,y\in\R.
\end{equation}
Then 
\begin{equation}\label{e:18'}
aE_{V_{1}}(\mu_{V_{1}})+(1-a)E_{V_{2}}(\mu_{V_{2}})\ge E_{V_{3}}(\mu_{V_{3}}).
\end{equation}
\end{theorem}

\begin{proof}
Take the (increasing) transportation map $\theta$ from $\mu_{V_{1}}$ into $\mu_{V_{2}}$.   This certainly exists as the measure $\mu_{V_{1}}$ has no atoms.  

Noticing that for any measure with finite logarithmic energy, we have the obvious equality 
\[
\int\log|x-y|\mu(dx)\mu(dy)=2\int_{x>y}\log(x-y)\mu(dx)\mu(dy).
\]  
 Using this we argue that    
\begin{align*}
\int & aV_{1}(x)+(1-a)V_{2}(\theta(x))\mu_{V_{1}}(dx)-2\iint_{x>y} \big  (a\log(x-y)+(1-a)\log(\theta(x)-\theta(y)) \big )\mu_{V_{1}}(dx)\mu_{V_{1}}(dy)\\ 
&\ge \int V_{3}\big (ax+(1-a)\theta(x) \big )\mu_{V_{1}}(dx)
 - 2\iint_{x>y} \log\big[(ax+(1-a)\theta(x))-(ay+(1-a)\theta(y)\big ]\mu_{V_{1}}(dx)\mu_{V_{1}}(dy) \\ 
& = E_{V_{3}}(\nu) \ge E_{V_{3}}(\mu_{V_{3}})
\end{align*}
where $\nu=(a {\rm{id}}+(1-a)\theta)_{\#}\mu_{V_{1}}$ and we used \eqref{e:17'} and the concavity of the logarithm on $(0,\infty)$. The proof is complete.
\end{proof}

\section{Random Matrices and a First Version of Poincar\'e Inequality}\label{s:p1}

In the next three sections, we investigate Poincar\'e type inequalities in the free (one-dimensional) context. We discuss two versions of it. The first one is suggested by large matrix approximations and the classical Poincar\'e inequality for strictly convex potentials, but will be proved directly. Recall first the classical Poincar\'e inequality
(cf. e.g. \cite {Bakry}, \cite{L2}, \cite {Villani2}, \cite {Wang}...).

\begin{theorem}  Let $\mu(dx)=e^{-W(x)}dx$ be a probability measure on $\R^{d}$ such that $W(x)-r|x|^{2}$ is convex.  Then for any compactly supported and smooth function $\phi:\R^{d}\to\R$, we have that 
\begin{equation}\label{e:20}
\int |\nabla \phi|^{2}d\mu\ge r \, \mathrm{Var}_{\mu}(\phi).
\end{equation}
\end{theorem}

Assume now that $V$ is a potential on $\R$ with enough growth at infinity.  Consider the matrix models on $\mathcal{H}_{n}$,  the space of Hermitian $n\times n$ matrices with the inner product $\langle A,B \rangle=Tr(AB^{*})$ and the probability measure given by 
\[
\p_{n}(dM)=\frac{1}{Z_{n}(V)} \, e^{-n\mathrm{Tr} (V(M))}dM
\]
where here $dM$ is the standard Lebesgue measure on $\mathcal{H}_{n}$. We have that for any bounded continuous function $F:\R\to\R$,
\begin{equation}\label{e:22}
\int \frac{1}{n} \, \mathrm{Tr} \big (F(M) \big )\p_{n}(dM)\xrightarrow[n\to\infty]{} \int F(x)\mu_{V}(dx).
\end{equation}

Assume in addition that $V(x)-\rho x^{2}$ is a convex function on $\R$.  Then, consider $\Phi(M)=\mathrm{Tr}\phi(M)$, where $\phi:\R\to\R$ is a compactly supported and  smooth function.  Notice that $\nabla \Phi(M)=\phi'(M)$ and thus $|\nabla \Phi(M)|^{2}=|\phi'(M)|^{2}=\mathrm{Tr}(\phi'(M)^{2})$.  
Since $n\mathrm{Tr}(V(M))-n\rho|M|^{2}$ is convex, we can apply Poincar\'e's inequality on $\mathcal{H}_{n}$ to obtain that 
\begin{equation}\label{e:21}
\int \mathrm{Tr} \big (\phi'(M)^{2}\big )\p_{n}(dM)
       \ge n \rho \, \mathrm{Var}_{\p_{n}} \big (\mathrm{Tr}(\phi(M)) \big ).
\end{equation}
The first term in this inequality (cf. equation \eqref{e:22}) converges to $\int \phi'(x)^{2}\mu_{V}(dx)$.  
To understand the second term in the above equation, notice that $\mathrm{Var}(\mathrm{Tr}(\phi(M)))=\E\left[\left( \mathrm{Tr}(\phi(M))- \E[ \mathrm{Tr}(\phi(M))]\right)^{2}\right]$.  The study of the asymptotic of the linear statistics, $\mathrm{Tr}(\phi(M))- \E[ \mathrm{Tr}(\phi(M))]$ in the literature of random matrix is known as ``fluctuations''.  From Johansson's paper \cite{J}, it is known that this is universal in the sense that the limit in distribution of the fluctuations is Gaussian and, at least in the case of polynomial $V$ (for which $V(x)-\rho x^{2}$ fulfills the conditions in there), the variance of the Gaussian limit depends only on the endpoints of the support of $\mu_{V}$.  Moreover, in the particular case of $V(x)=2 x^{2}$, the variance of the distribution was computed for example in \cite{Pastur} and \cite{J} as
\begin{equation}\label{e:200}
\frac{1}{2\pi^{2}}\int_{-1}^{1}\int_{-1}^{1}\left(\frac{\phi(t)-\phi(s)}{t-s}\right)^{2}\frac{1-ts}{\sqrt{1-t^{2}}\sqrt{1-s^{2}}} \, dtds.
\end{equation}

 This variance is interpreted in \cite{CD1} in terms of the number operator of the arcsine law. We will come back to this aspect in Section~\ref{s:spectral}.

Dividing the inequality in equation \eqref{e:21} by $n$ and taking the limit when $n\to\infty$, these heuristics (after a simple rescaling) suggest the following result.  

\begin{theorem}\label{t:6}
Assume that $V(x)-\rho x^{2}$ is convex for some $\rho>0$.  Then for any smooth function $\phi$, one has that
\begin{equation}\label{e:202}
\int \phi'(x)^{2}\mu_{V}(dx)\ge \frac{\rho}{2\pi^{2}}\int_{a}^{b}\int_{a}^{b}\left(\frac{\phi(x)-\phi(y)}{x-y}\right)^{2}\frac{-2ab+(a+b)(x+y)-2xy}{2\sqrt{(x-a)(b-x)}\sqrt{(y-a)(b-y)} } \, dxdy.
\end{equation}
where $\mathrm{supp}(\mu_{V})=[a,b]$.  Equality is attained for $V(x)=\rho (x-\alpha)^{2}+\beta$ and $\phi(x)=c_{1}+c_{2}x$ for some constants $c_{1},c_{2}$.  
\end{theorem}

The reader may wonder if the numerator in the second fraction of \eqref{e:202} is nonnegative.  This is so because 
\[
-2ab+(a+b)(x+y)-2xy=2\left( \left( \frac{b-a}{2}\right)^{2}-\left(x-\frac{a+b}{2}\right)\left(y-\frac{a+b}{2}\right)   \right)\ge0
\]
for any $x,y\in[a,b]$.

\begin{proof}
Using a simple rescaling we may assume without loss of generality that $a=-1$ and $b=1$ and the inequality we have to show reduces to
\begin{equation}\label{e:23}
\int \phi'(x)^{2}\mu_{V}(dx)\ge \frac{\rho}{2\pi^{2}}\int_{-1}^{1}\int_{-1}^{1}\left(\frac{\phi(x)-\phi(y)}{x-y}\right)^{2}\frac{1-xy}{\sqrt{1-x^{2}}\sqrt{1-y^{2}}} \, dxdy.
\end{equation}
Then, based on equation \eqref{e:16}, we have that 
\[
g(x)=\frac{\sqrt{1-x^{2}}}{2\pi^{2}}\int_{-1}^{1}\frac{V'(y) -V'(x)}{\sqrt{1-y^{2}}(y-x)} \, dy.
\]
From the convexity of $V(x)-\rho x^{2}$, we learn that $\frac{V'(y)-V'(x)}{y-x}\ge2\rho$ and thus that
\begin{equation}\label{e:25}
g(x)\ge\frac{\rho}{\pi}\sqrt{1-x^{2}} \, , 
\end{equation}
which implies
\[
\int \phi'(x)^{2}\mu_{V}(dx)\ge \frac{\rho}{\pi}\int_{-1}^{1}\phi'(x)^{2}\sqrt{1-x^{2}} \, dx.
\]
Therefore it is enough to check that
\begin{equation}\label{e:91}
\int_{-1}^{1}\phi'(x)^{2}\sqrt{1-x^{2}} \, dx\ge \frac{1}{2\pi}\int_{-1}^{1}\int_{-1}^{1}\left(\frac{\phi(x)-\phi(y)}{x-y}\right)^{2}\frac{1-xy}{\sqrt{1-x^{2}}\sqrt{1-y^{2}}} \, dxdy 
\end{equation}
for any smooth $\phi $.
Now, we make the change of variables $x=\cos t$ to justify 
\[
\int_{-1}^{1} \phi'(x)^{2}\sqrt{1-x^{2}} \, dx=\int_{0}^{\pi}\phi'(\cos t)^{2}\sin^{2}(t)dt=\int_{0}^{\pi}\psi'(t)^{2}dt
\]
where $\psi(t)=\phi(\cos t)$.   

On the other hand, using the change of variable $x=\cos t$, $y=\cos s$ on the right hand side, inequality \eqref{e:91} becomes
\begin{equation}\label{e:24}
\int_{0}^{\pi}\psi'(t)^{2}dt \ge \frac{1}{2\pi}\int_{0}^{\pi}\int_{0}^{\pi}\left(\frac{\psi(t)-\psi(s)}{\cos t-\cos s}\right)^{2}(1-\cos t\cos s) dtds.
\end{equation}
To show this, we write $\psi(t)=\sum_{k=0}^{\infty}a_{k}\cos kt$ and then, because $\psi$ is a smooth function, we can differentiate term by term to get 
$\psi'(t)=-\sum_{k=1}^{\infty}ka_{k}\sin kt$, therefore
\[
\int_{0}^{\pi}\psi'(t)^{2}dt=\frac{\pi}{2}\sum_{k=1}^{\infty}k^{2}a_{k}^{2}
\]
and 
\[
\int_{0}^{\pi}\int_{0}^{\pi}\left(\frac{\psi(t)-\psi(s)}{\cos t-\cos s}\right)^{2}(1-\cos t\cos s) dtds=\sum_{k,l=1}^{\infty}a_{k}a_{l}\int_{0}^{\pi}\int_{0}^{\pi}\frac{(\cos k t-\cos k s)(\cos l t -\cos l s )(1-\cos t\cos s)}{(\cos t-\cos s)^{2}}dtds.
\]
To compute the integrals on the right hand side of the above equation, we take the generating function of these numbers  and with a little algebra one can show that
\begin{equation}\label{e:701}
\begin{split}
\sum_{k,l=1}^{\infty}&u^{k}v^{l}\int_{0}^{\pi}\int_{0}^{\pi}\frac{(\cos k t-\cos k s)(\cos l t  -\cos l s )(1-\cos t\cos s)}{(\cos t-\cos s)^{2}} \, dtds \\ 
&=\int_{0}^{\pi}\int_{0}^{\pi}\frac{(u-u^{3})(v-v^{3})(1-\cos t\cos s)}{(1+u^{2}-2u\cos t)(1+u^{2}-2u\cos s)(1+v^{2}-2v\cos t)(1+v^{2}-2v\cos s)} \, dtds \\
& = \frac{\pi^{2}uv}{(1-uv)^{2}}=\pi^{2}\sum_{k=1}^{\infty}k u^{k}v^{k}
\end{split}
\end{equation}
for all $u,v\in(-1,1)$.  The last integral can be computed as follows.   First use partial fractions to justify
\[
\int_{0}^{\pi}\frac{(A+B\cos t)dt}{(1+u^{2}-2u\cos t)(1+v^{2}-2v\cos t)}=\int_{0}^{\pi}\frac{Cdt}{1+u^{2}-2u\cos t}+\int_{0}^{\pi}\frac{Ddt}{1+v^{2}-2v\cos t}=\frac{C/2}{1-u^{2}}+\frac{D/2}{1-v^{2}}
\]
where the constants $C,D$ are linear combinations of $A$ and $B$.   Further, taking $A=1$ and $B=-\cos s$ and repeating once more the partial fractions argument, one can cary out the proof of \eqref{e:701}. 

The main consequence of the above calculation is that 
\[
\int_{0}^{\pi}\int_{0}^{\pi}\frac{(\cos k t-\cos k s)(\cos l t  -\cos l s )(1-\cos t\cos s)}{(\cos t-\cos s)^{2}} \, dtds=\pi^{2}k\delta_{kl}
\]
and that
\begin{equation}\label{e:110}
\int_{0}^{\pi}\int_{0}^{\pi}\left(\frac{\psi(t)-\psi(s)}{\cos t-\cos s}\right)^{2}(1-\cos t\cos s) dtds=\pi^{2}\sum_{k=1}^{\infty}ka_{k}^{2}.  
\end{equation}
Therefore inequality \eqref{e:24} becomes equivalent to 
\[
\frac{\pi}{2}\sum_{k=1}^{\infty}k^{2}a_{k}^{2}\ge\frac{\pi}{2}\sum_{k=1}^{\infty}ka_{k}^{2}
\]
which is obviously true.  Notice that equality in this inequality is attained for the case $a_{k}=0$ for all $k\ge2$ and arbitrary $a_{1}$.  This corresponds to the case $\psi(t)=c_{1}+c_{2}\cos t$ or $\phi(x)=c_{2}x+c_{1}$ for some $c_{1},c_{2}$.    

Finally we point out that equality in \eqref{e:23} is attained  if the equality is attained in \eqref{e:25} and \eqref{e:24}.  From there one can easily see from rescaling  that equality in \eqref{e:202} is attained for $V(x)=\rho(x-\alpha)^{2}+\beta$ and $\phi(x)=c_{1}+c_{2}x$.  The proof of Theorem \ref{t:6} is complete. \qedhere
\end{proof}

In the above proof we showed a direct calculation for equation \eqref{e:110} which is natural in the course of the above proof.  However, there is another way of looking at it which will appear below in Section~\ref{s:spectral} as the kernel of the number operator.

\section{A Second Version of Poincar\'e Inequality}

The second version of the Poincar\'e inequality is motivated by the free calculus and the noncommutative derivative. It was already investigated by Biane \cite{Biane2} for the case of the semicircular law.  

\begin{definition} For a given probability measure $\mu$ on $\R$, we say that it satisfies a Poincar\'e inequality if there is a constant $C>0$ such that 
\begin{equation}\label{1}
\iint \left( \frac{\phi(x)-\phi(y)}{x-y}\right)^{2}\mu(dx)\mu(dy)\ge C\, \var_\mu (\phi) \quad\text{for every}\quad\phi\in C_{0}^{1}(\R).
\end{equation} 
By the best constant we mean the largest  $C>0$ for which the above inequality is satisfied and we denote it by $\mathrm{Poin}(\mu)$ or $\lambda_{1}(\mu)$ or $SG(\mu)$.
\end{definition}

In the noncommutative setting for a given function $\phi$, we can think of $D\phi(x,y)=\frac{\phi(x)-\phi(y)}{x-y}$ as the noncommutative derivative of $\phi$.  As pointed out by Voiculescu in \cite{Voiculescu}, this is the unique map $D:C\langle x\rangle \to C \langle x\rangle \otimes C\langle x\rangle$ such that 
\begin{enumerate}
\item $D1=0$
\item $D(fg)=D(f)g+fD(g)$ for any $f,g\in C\langle x\rangle$. 
\end{enumerate}

First we collect a couple of obvious properties of the  Poincar\'e constant.  

\begin{proposition}
\begin{enumerate}
\item For any $a\ne 0$, 
\[
\mathrm{Poin} \big ((ax+b)_{\#}\mu \big )=\frac{1}{a^{2}} \, \mathrm{Poin}(\mu)
\]
 where here and elsewhere, for a given function $f:\R\to\R$, $f_{\#}\mu$ is the push forward measure given by $(f_{\#}\mu)(A)=\mu(f^{-1}(A))$.
\item If $f:\R\to\R$ is a differential map such that $|f'(x)|\ge c>0$ for all $x\in\R$, then 
\[
\mathrm{Poin}(\mu)\ge c^{2} \, \mathrm{Poin}(f_{\#}\mu).
\]
\item If $\{\mu_{n}\}_{n\ge1}$  is a sequence of probability measures which converges weakly to $\mu$, then 
\[
\mathrm{Poin}(\mu)\ge\limsup_{n\to\infty}\mathrm{Poin}(\mu_{n}).
\]
\end{enumerate}

\end{proposition}

Next we describe some bounds for the Poincar\'e constant.  

\begin{theorem}\label{t:7}
Assume that  the measure $\mu$ has compact support and is not concentrated at one point.  
Then $\mu$ satisfies a Poincar\'e inequality with
\begin{equation}\label{2}
\frac{2}{d^{2}(\mu)}\le \mathrm{Poin}(\mu)\le \frac{1}{\var(\mu)}
\end{equation}
where $d(\mu)=\mathrm{diam}(\mathrm{supp}(\mu))$ is the diameter of the support of $\mu$
and $\var (\mu) = \int x^2 \mu(dx) - \big ( \int x \mu(dx) \big)^2$.
Equality on the left in \eqref{2} is attained only for the case 
\[ 
\mu=\alpha\delta_{a}+(1-\alpha)\delta_{b},\quad a<b,\quad 0<\alpha<1.
\] 
Equality on the right of \eqref{2} is attained only for the case of a semicircular law ($a\in\R$, $r>0$) 
\[
\mu(dx)=\frac{1}{2\pi r^{2}}\mathbbm{1}_{[a-2r,a+2r]}(x)\sqrt{4r^{2}-(x-a)^{2}} \, dx.
\]
In addition, assume that $V$ is a $C^{2}$ potential on $\R$ such that for some integer $p$ and real  $\rho>0$, $V(x)-\rho x^{2p}$,  is convex and $\mu$ is the minimizer of 
\[
\int V(x)\mu(dx)-\iint \log|x-y| \, \mu(dx)\mu(dy)
\]
over all probability measures of $\R$.   Then  
\begin{equation}\label{5}
\frac{\left( p \rho {2p\choose p}\right)^{\frac{1}{p}}}{8}\le \mathrm{Poin}(\mu).
\end{equation}  
In particular if $p=1$, we get that $\frac{\rho}{4}\le \mathrm{Poin}(\mu)$.

\end{theorem}

\begin{proof}[\bf Proof]
For a given function $\phi\in C_{0}^1(\R)$,  the left hand side of  \eqref{2} follows from
\begin{align}\label{3}
 \var_\mu(\phi) = \frac{1}{2}\iint (\phi(x)-\phi(y))^{2}\mu(dx)\mu(dy) 
& \notag =  \frac{1}{2}\iint (x-y)^{2}\left(\frac{\phi(x)-\phi(y)}{x-y}\right)^{2}\mu(dx)\mu(dy)   \\ 
& \le \frac{d^{2}(\mu)}{2} \iint \left(\frac{\phi(x)-\phi(y)}{x-y}\right)^{2}\mu(dx)\mu(dy).
\end{align}
The right hand side of \eqref{2} follows  from \eqref{1}  for a  $\phi\in C_{0}^{1}(\R)$ such that  $\phi(x)=x$ on the support of $\mu$.  

For measures  $\mu=\alpha\delta_{a}+(1-\alpha)\delta_{b}$, condition \eqref{1} is equivalent to 
\[
C\alpha(1-\alpha) \big (\phi(b)-\phi(a)\big )^{2}
\le \alpha^{2}(\phi'(a))^{2}+(1-\alpha)^{2}(\phi'(b))^{2}+2\alpha(1-\alpha)\left( \frac{\phi(b)-\phi(a)}{b-a} \right)^{2} \quad \text{for any} \: \: \phi\in C_{0}^{1}(\R).
\]
Since for any function $\phi\in C_{0}^{\infty}(\R)$ we can find another function $\psi\in C_{0}^{1}(\R)$ so that $\phi(a)=\psi(a)$ and $\phi(b)=\psi(b)$ and $\psi(a)=0$, $\psi(b)=0$, this  is also equivalent to 
\[
C\alpha(1-\alpha)\big (\psi(b)-\psi(a)\big )^{2}
\le 2\alpha(1-\alpha)\left( \frac{\psi(b)-\psi(a)}{b-a} \right)^{2} \quad \text{for any}\: \:  \psi\in C_{0}^{1}(\R). 
\]
This amounts to $C\le 2/(b-a)^{2}$ and therefore, in this case, $\mathrm{Poin}(\mu)=\frac{2}{d^{2}(\mu)}$.
 
 Conversely,  if $\mu$ is a measure so that $\mathrm{Poin}(\mu)=\frac{2}{d^{2}(\mu)}$, then, for  $1>\epsilon>0$, there is a function $\phi_{\epsilon}\in C_{0}^{1}(\R)$ such that 
 \[
\left(\frac{2}{d^{2}(\mu)}+\epsilon^{2}\right) \var _\mu (\phi_\epsilon) >   \iint \left( \frac{\phi_{\epsilon}(x)-\phi_{\epsilon}(y)}{x-y}\right)^{2}\mu(dx)\mu(dy). 
 \]
Without loss of generality we can assume that $0=\inf \mathrm{supp}(\mu)$, $1=\sup \mathrm{supp}(\mu)$ and  $\int\phi_{\epsilon}d\mu=0$, $\int \phi_{\epsilon}^{2}d\mu=1$ where we recall that $\mathrm{supp}(\mu)$ stands for the support of $\mu$.  In this case, the above inequality implies
\begin{align*}
2+\epsilon^{2} &\ge  \iint_{|x-y|\ge1-\epsilon} \left( \frac{\phi_{\epsilon}(x)-\phi_{\epsilon}(y)}{x-y}\right)^{2}\mu(dx)\mu(dy) +  \iint_{|x-y|<1-\epsilon} \left( \frac{\phi_{\epsilon}(x)-\phi_{\epsilon}(y)}{x-y}\right)^{2}\mu(dx)\mu(dy)  \\
& \ge \iint_{|x-y|\ge1-\epsilon} \big( \phi_{\epsilon}(x)-\phi_{\epsilon}(y)\big)^{2}\mu(dx)\mu(dy) +  \frac{1}{(1-\epsilon)^{2}} \iint_{|x-y|<1-\epsilon} \big( \phi_{\epsilon}(x)-\phi_{\epsilon}(y)\big)^{2}\mu(dx)\mu(dy) \\ 
& = -\frac{\epsilon(2-\epsilon)}{(1-\epsilon)^{2}}\iint_{|x-y|\ge1-\epsilon} \big( \phi_{\epsilon}(x)-\phi_{\epsilon}(y)\big)^{2}\mu(dx)\mu(dy) +\frac{2}{(1-\epsilon)^{2}} \, ,
\end{align*}
which results with 
\begin{equation}\label{e:500}
\iint_{|x-y|\ge1-\epsilon} \big( \phi_{\epsilon}(x)-\phi_{\epsilon}(y)\big)^{2}\mu(dx)\mu(dy) 
\ge 2-\frac{\epsilon(1-\epsilon)^{2}}{2-\epsilon} \, .
\end{equation}
Now, 
\begin{equation}\label{e:501}
\begin{split}
 \iint_{|x-y|\ge1-\epsilon} \big( \phi_{\epsilon}(x)-\phi_{\epsilon}(y)\big)^{2}\mu(dx)\mu(dy) 
 & \le  \iint_{\substack{|x-1/2|\ge 1/2-\epsilon \\ |y-1/2|\ge 1/2-\epsilon }} \big( \phi_{\epsilon}(x)-\phi_{\epsilon}(y)\big)^{2}\mu(dx)\mu(dy)\\ 
 & \le 2\mu \big (|x-1/2|\ge 1/2-\epsilon \big ).
 \end{split}
\end{equation}
Thus \eqref{e:500} and \eqref{e:501} give 
\[
\mu \big (|x-1/2|\ge 1/2-\epsilon \big )\ge 1-\frac{\epsilon(1-\epsilon)^{2}}{4-2\epsilon} 
\quad \text{for any}\: \: 1>\epsilon>0.
\]
This shows that $\mu((0,1))=0$ and therefore $\mu=\alpha \delta_{0}+(1-\alpha)\delta_{1}$. 

The other extreme case of inequality \eqref{2} is contained in Biane's paper \cite{Biane2} in the more general context of several noncommutative variables.  For completeness we will provide here a selfcontained proof.   In the first place, using Proposition~\ref{1}, we may assume that 
\[
\mu(dx)=\frac{1}{2\pi }\mathbbm{1}_{[-2,2]}(x)\sqrt{4-x^{2}} \, dx
\]
is the semicircular law on $[-2,2]$.
Take $U_{n}$ to be the  Chebyshev polynomials of second kind defined by $U_{n}(\cos(\theta))=\frac{\sin(n+1)\theta}{\sin\theta}$.  With this choice, we have that $U_{n}(\frac{x}{2})$ are the orthogonal polynomials with respect to $\mu$.   The generating function of $U_{n}$ is given by
\[
\sum_{n=0}^{\infty}r^{n}U_{n}(x)=\frac{1}{1-2rx+r^{2}}\quad \text{for}\quad |x|,|r|<1
\]
from which one gets
\[
\sum_{n=0}^{\infty}r^{n}\frac{U_{n}(x)-U_{n}(y)}{x-y}=\frac{2r}{(1-2rx+r^{2})(1-2ry+r^{2})}=2\sum_{n=0}^{\infty}r^{n}\sum_{k=0}^{n-1}U_{k}(x)U_{n-1-k}(y),
\]
and then
\begin{equation}\label{4}
\frac{U_{n}(x)-U_{n}(y)}{x-y}=2\sum_{k=0}^{n-1}U_{k}(x)U_{n-1-k}(y).
\end{equation}
Now, for a given $\phi\in C_{0}^{1}(\R)$, we can write in $L^{2}(\mu)$ sense,
\[
\phi(x)=\sum_{n=0}^{\infty}\alpha_{n}U_{n}\left(\frac{x}{2}\right),
\]
yielding from orthogonality and \eqref{4} that
\[
\var _\mu( \phi) =
\int \phi^{2}d\mu -\left( \int \phi d\mu\right)^{2}=\sum_{n=1}^{\infty}\alpha_{n}^{2} \quad \text{and}\quad \iint \left( \frac{\phi(x)-\phi(y)}{x-y}\right)^{2}\mu(dx)\mu(dy) = \sum_{n=1}^{\infty} n \alpha_{n}^{2}.
\]
It follows that in this case $\mathrm{Poin} (\mu)=1=1/\var(\mu)$ and equality is attained only for $\phi(x)=c_{1}+c_{2}U_{1}(x)=c_{1}+c_{2}x$ for some constants $c_{1},c_{2}$.  

To prove the converse, take a compactly supported measure $\mu$ and assume that $\int x \mu(dx)=0$ and $\int x^{2}\mu(dx)=1$.   In order to show that $\mu$ is the semicircular distribution,
it suffices to show that $\int U_{n}\left( \frac{x}{2} \right)\mu(dx)=0$ for all $n\ge1$. We use
induction to this task. Assuming true for $U_{1},U_{2},\dots, U_{n}$, and using $U_{n+1}(x)=2xU_{n}(x)-U_{n-1}(x)$, we need to show that $xU_{n}\left(\frac{x}{2} \right)$ integrates to $0$ against $\mu$.   Applying Poincar\'e's inequality to $U_{n}\left( \frac{x}{2}\right)+rU_{1}\left( \frac{x}{2}\right)$ together with the induction hypothesis and equation \eqref{4}, we get that for any $r\in\R$, 
\[
\int U_{n}^{2}\left( \frac{x}{2}\right)\mu(dx)  +r\int xU_{n}\left( \frac{x}{2}\right)\mu(dx)   \le  \iint \left(\frac{U_{n}\left( \frac{x}{2}\right)-U_{n}\left( \frac{y}{2}\right)}{x-y} \right)^{2}\mu(dx)\mu(dy),
\]
which implies that $\int xU_{n}\left( \frac{x}{2}\right)\mu(dx)=0$.  

In the case of the equilibrium measure of a convex potential $V$, we have the support of the measure consists of one interval $[a,b]$ and $a$, $b$ solve the system (cf. equation \eqref{e:0})
\[
\frac{1}{2\pi} \int_{a}^{b}V'(x)\sqrt{\frac{x-a}{b-x}} \, dx=1 \quad \text{and}\quad \frac{1}{2\pi} \int_{a}^{b}V'(x)\sqrt{\frac{b-x}{x-a}} \, dx=-1.  
\]
If we denote $c=(b-a)/2$ and $\beta=(a+b)/2$, the system above can be rewritten in terms of $\beta$ and $c$ as 
\[
\frac{c}{2\pi} \int_{-1}^{1}V'(\beta+ct) \, \frac{1+t}{\sqrt{1-t^{2}}} \, dt=1 \quad
 \text{and}\quad \frac{c}{2\pi} \int_{-1}^{1}V'(\beta+ct) \, \frac{1-t}{\sqrt{1-t^{2}}} \, dt=-1
\]
which is equivalent to 
\[
\frac{c}{2\pi} \int_{-1}^{1}V'(\beta+ct) \, \frac{t}{\sqrt{1-t^{2}}} \, dt=1 \quad \text{and}\quad  \int_{-1}^{1}V'(\beta+ct) \, \frac{1}{\sqrt{1-t^{2}}} \, dt=0.
\]
Since $V$ is $C^{2}$ the first equation can be integrated by parts to get that 
\[
\frac{c^{2}}{2\pi} \int_{-1}^{1}V''(\beta+ct)\sqrt{1-t^{2}} \, dt=1.  
\]
On the other hand we know that $V''(x)\ge 2p(2p-1)\rho x^{2p-2}$, hence 
\begin{align*}
1&\ge \frac{2p(2p-1)\rho c^{2}}{2\pi}\int_{-1}^{1}(ct+\beta)^{2p-2}\sqrt{1-t^{2}} \,dt \\
 &\ge\frac{2p(2p-1)\rho c^{2p}}{2\pi}\int_{-1}^{1}t^{2p-2}\sqrt{1-t^{2}} \, dt \\ 
&= \frac{p(2p-1)\rho c^{2p}{2p \choose p}}{4^{p}(2p-1)}\\ 
&=\frac{p \rho {2p\choose p} c^{2p}}{4^{p}} \, .
\end{align*}
This yields 
\[
c\le 2\left( m \rho {2p\choose p}\right)^{-\frac{1}{2p}}.
\]
Finally, because $d(\mu)=b-a=2c$, we arrive at \eqref{5}.\qedhere
\end{proof}

To conclude this section, we present an inequality which relates the equilibrium measure of a strong convex potential and the arcsine  law.  

\begin{theorem}
Assume that $V(x)-\rho x^{2}$ is a convex for some $\rho>0$ and the equilibrium measure $\mu_{V}$ has support $[a,b]$.  Let $\mathrm {arcsine}_{a,b}=\mathbbm{1}_{[a,b]}(x)\frac{1}{\pi\sqrt{(b-x)(x-a)}}dx$ be the arcsine law with support $[a,b]$. Then for any smooth function supported on $[a,b]$, 
\begin{equation}
\int \phi'(x)^{2}\mu_{V}(dx)\ge \rho \, \var_{\mathrm {arcsine}_{a,b}}(\phi),
\end{equation}
where the variance is considered with respect to the $\mathrm {arcsine}_{a,b}$ law.  
\end{theorem}

\begin{proof}  It suffices to deal with the case $a=-1$, $b=1$, the rest following by simple rescaling.  
Recall that  in the proof of Theorem \ref{t:6}, we use convexity to get that the density $g(x)$ of $\mu_{V}$ satisfies
$ g(x)\ge \frac{\rho}{\pi}\sqrt{1-x^{2}}$. Thus the proof reduces to 
\begin{equation}\label{e:122}
\frac{1}{\pi}\int_{-1}^{1} \phi'(x)^{2}\sqrt{1-x^{2}}(dx)\ge  \var_{\mathrm {arcsine}} (\phi ) .
\end{equation}
For this, write $\phi=\sum_{n=0}^{\infty}\alpha_{n}T_{n}(x)$  the expansion of $\phi$ in terms of  Chebyshev polynomials of the first kind.  Now, $T_{n}'=nU_{n-1}$ and thus the above inequality reduces to the obvious inequality $
\sum_{n=1}^{\infty}n^{2}\alpha_{n}^{2}\ge \sum_{n=1}^{\infty}\alpha_{n}^{2}$.
\qedhere
\end{proof}

We will actually see below that inequality \eqref{e:122} is simply the spectral gap for the Jacobi operator associated to the arcsine law.

\section{Poincar\'e Inequalities and Jacobi Operators}\label{s:spectral}

In this section we show how the two versions of the Poincar\' e inequalities can be viewed as spectral gaps for some Jacobi operators. This discussion is mainly driven from the work \cite{CD1} by Cabanal-Duvillard and his interpretation of the variance in \eqref{e:200} in terms of the number operator of the Jacobi operator associated to the arcsine law. This viewpoint allows for an unified perspective of the Poincar\'e inequalities presented in the preceding sections.

For our purpose we consider here the Jacobi operators given,  for smooth functions on $(-1,1)$, 
by
\begin{equation}\label{e:jac}
L_{\lambda}f(x)=-(1-x^{2})f''(x)+(2\lambda+1) x f'(x)
\end{equation}
for $\lambda\ge0$.  We consider the Gegenbauer polynomials $C^{\lambda}_{n}$, $\lambda>0$,  defined by the generating function
\[
\sum_{n=0}^{\infty}r^{n}C_{n}^{\lambda}(x)=\frac{1}{(1-rx+r^{2})^{\lambda}}.
\]
For $\lambda=0$ we set $C_{n}^{\lambda}(x)=T_{n}(x)/n$, $n\ge1$, where $T_{n}$ are the Chebyshev polynomials of the first kind.   

It is known that $C_{n}^{\lambda}$ are eigenfunctions of $L_{\lambda}$, with eigenvalue $n(n+2\lambda)$, i.e. 
\[
L_{\lambda}C^{\lambda}_{n}=n(n+2\lambda)C^{\lambda}_{n}
\]
On the other hand the Gegenbauer polynomials are orthogonal with respect to the probability measure  
\[
\nu_{\lambda}=\frac{2^{2\lambda} \Gamma^{2}(\lambda+1)}{\pi\Gamma(2\lambda+1)}
    \, \mathbbm{1}_{[-1,1]}(x)(1-x^{2})^{\lambda-1/2}.
\] 
Notice that in the case of $\lambda=0$, this becomes the arcsine law and for $\lambda=1$, this is the semicircular law, while for $\lambda=1/2$, this becomes the uniform measure on $[-1,1]$.  

Take now the normalized Gegenbauer polynomials $\phi^{\lambda}_{n}=G^{\lambda}_{n}/\sqrt{c^{\lambda}_{n}}$, where $c_{n}^{\lambda}=\int G_{n}^{\lambda}(x)^{2}\nu_{\lambda}(dx)$.  Then $\phi_{n}^{\lambda}$ form an orthonormal basis of $L^{2}(\nu_{\lambda})$ and thus the operator $L_{\lambda}$ is diagonalized in this basis.  
Consider $N_{\lambda}$ to be the counting number operator with respect to the basis $\phi^{\lambda}_{n}$, i.e. 
\begin{equation}\label{e:count}
N_{\lambda}\phi^{\lambda}_{n}=n\phi_{n}^{\lambda}.
\end{equation}  
This implies that 
$L_{\lambda}=N_{\lambda}^{2}+2\lambda N_{\lambda}$.
Therefore we have the following two inequalities
\begin{equation}\label{e:29}
L_{\lambda}\ge (2\lambda+1)N_{\lambda}\quad\text{and}\quad N_{\lambda}\ge 1-P_{\lambda}
\end{equation}
where $P_{\lambda}$ here stands for the projection on constant functions in $L^{2}(\nu_{\lambda})$.  In other words, $P_{\lambda}\phi=\int \phi\nu_{\lambda}$.  

Notice that equation \eqref{e:29} include two statements.  The first one is the comparison of $L$ and $N$, with the spectral gap $2\lambda+1$ while  the second one is the spectral gap of the counting number operator with the spectral gap $1$.  In the sequel we want to translate these spectral gaps in terms of Poincar\'e type inequality.  For this matter we need to find the kernel of the operator $N$.   

Then we have for any function in the domain of definition of  $L_{\lambda}$, that $\phi=\sum_{n=0}^{\infty}\alpha_{n}\phi_{n}^{\lambda}$, and then 
\[
\langle L\phi ,\phi \rangle_{L^{2}(\nu_{\lambda})}=\sum_{n=0}^{\infty}n(n+2\lambda)\alpha_{n}^{2}.
\]
On the other hand, using integration by parts, we can justify that
 \[
\langle L\phi ,\phi \rangle_{L^{2}(\nu_{\lambda})}= \int \phi L_{\lambda}\phi d\nu_{\lambda}=\int\phi'(x)^{2}(1-x^{2})\nu_{\lambda}(dx). 
 \]
For the number operator, we have that 
\[
\int \phi N_{\lambda}\phi d\nu_{\lambda}=\sum_{n=0}^{\infty}n\alpha_{n}^{2}=\lim_{r\uparrow 1}\sum_{n=0}^{\infty}nr^{n-1}\alpha_{n}^{2}.
\]
Now, for $-1<r<1$, 
\[
\sum_{n=0}^{\infty}nr^{n-1}\alpha_{n}^{2}=\iint\phi(x)\phi(y)\sum_{n=0}^{\infty}n r^{n-1}\phi_{n}^{\lambda}(x)\phi_{n}^{\lambda}(y)\nu_{\lambda}(dx)\nu_{\lambda}(dy).
\]
Furthermore, since $\int \phi_{n}^{\lambda}d\nu_{\lambda}=0$ for $n\ge1$, we also obtain that $\iint\phi^{2}(x)\phi_{n}^{\lambda}(y)\nu_{\lambda}(dx)\nu_{\lambda}(dy)=0$ for $n\ge 0$ and thus, denoting $K_{\lambda}(r,x,y)=-\sum_{n=0}^{\infty}nr^{n-1}\phi_{n}^{\lambda}(x)\phi_{n}^{\lambda}(y)$,
\[
\iint\phi(x)\phi(y)\sum_{n=0}^{\infty}n r^{n-1}\phi_{n}^{\lambda}(x)\phi_{n}^{\lambda}(y)\nu_{\lambda}(dx)\nu_{\lambda}(dy)=\frac{1}{2}\iint (\phi(x)-\phi(y))^{2}K_{\lambda}(r,x,y)\nu_{\lambda}(dx)\nu_{\lambda}(dy).  
\]
The following formula is essentially due to Watson \cite{Watson} and valid for $\lambda>0$, 
\[
\sum_{n=0}^{\infty}r^{n}\phi_{n}^{\lambda}(x)\phi_{n}^{\lambda}(y)=\frac{(1-r^{2})\Gamma(2\lambda)}{2^{2\lambda-1}\Gamma^{2}(\lambda)}\int_{-1}^{1}\frac{(1-z^{2})^{\lambda-1}}{(1-2r(xy+z\sqrt{(1-x^{2})(1-y^{2})})+r^{2})^{1+\lambda}} \, dz. 
\]
For $\lambda=0$, we have to deal with the Chebyshev polynomials of the first kind  which was more or less what appeared in the proof of Theorem \ref{t:6}.  For this case, we have that  (denoting $x=\cos t$ and $y=\cos s$),
\[
\sum_{n=0}^{\infty} \frac{r^{n}}{c_n} \, T_{n}(x)T_{n}(y)=\frac{1-r\cos(t+s)}{1-2r\cos(t+s)+r^{2}}
+\frac{1-r\cos(t-s)}{1-2r\cos(t-s)+r^{2}}
\]
where $c_{n}=\int T_{n}^{2}d\nu_{0}=1$ for $n=0$ and $1/2$ otherwise.    

Thus, we obtain, after differentiation with respect to $r$ and then limit over $r\uparrow 1$, that
\begin{equation}\label{e:40}
  K_{\lambda}(x,y)=\lim_{r\uparrow 1}K_{\lambda}(r,x,y)=\begin{cases}
\displaystyle \frac{\Gamma(2\lambda)}{2^{3\lambda-1}\Gamma^{2}(\lambda)}\int_{-1}^{1}\frac{(1-z^{2})^{\lambda-1}}{\left(1-xy-z\sqrt{(1-x^{2})(1-y^{2})}\right)^{1+\lambda}} \, dz, & \lambda>0\\
\displaystyle\frac{1-xy}{(x-y)^{2}},  & \lambda=0\\
\displaystyle \frac{1}{2(x-y)^{2}}, & \lambda=1.  
\end{cases}
\end{equation}
The integrand is not a rational function. In some cases, it is algebraic 
since $\lambda\ge0$ need not be an integer.

To reveal  the singularity of this kernel, we make the change of variable 
\[
1-xy-z\sqrt{(1-x^{2})(1-y^{2})}=t \Big (1-xy-\sqrt{(1-x^{2})(1-y^{2})} \Big ).
\] 
Then, after simple algebraic manipulations, setting $f_{\lambda}:(0,1)\to \R$,
\[
f_{\lambda}(u)=\int_{1}^{1/u}\frac{\left[(t-1)(1-ut)\right]^{\lambda-1}}{t^{\lambda+1}} \, dt ,
\]
and 
\begin{equation}\label{e:41}
H_{\lambda}(x,y)=
\begin{cases}
\displaystyle \frac{\Gamma(2\lambda)\left( 1-xy+\sqrt{(1-x^{2})(1-y^{2})} \right)^{\lambda}}{2^{3\lambda-1}\Gamma^{2}(\lambda) \left((1-x^{2})(1-y^{2})\right)^{\lambda-1/2}}f_{\lambda}\left(\frac{(x-y)^{2}}{\left(1-xy+\sqrt{(1-x^{2})(1-y^{2})}\right)^{2}}\right), & \lambda>0 \\
\displaystyle 1-xy, & \lambda=0,\\
\displaystyle\frac{1}{2}\, ,  &\lambda=1, 
\end{cases}
\end{equation}
we can rewrite equation \eqref{e:40} for $|x|,|y|<1$ as 
\begin{equation}
K_{\lambda}(x,y)=\frac{H_{\lambda}(x,y)}{(x-y)^{2}}
\end{equation}
where $H_{\lambda}(x,y)$ is a continuous function of $x,y\in[-1,1]$.  

Now, from  \eqref{e:29}, we obtain the following result.  

\begin{theorem}
For any $\lambda\ge0$, one has for all $\lambda\ge0$ and any $\phi\in C^{1}([-1,1])$, that 
\begin{equation}\label{e:42}
\int\phi'(x)^{2}(1-x^{2})\nu_{\lambda}(dx)\ge \frac{2\lambda+1}{2} \iint \left(\frac{\phi(x)-\phi(y)}{x-y} \right)^{2}H_{\lambda}(x,y)\nu_{\lambda}(dx)\nu_{\lambda}(dy).
\end{equation}
and
\begin{equation}\label{e:43}
\iint \left(\frac{\phi(x)-\phi(y)}{x-y} \right)^{2}H_{\lambda}(x,y)\nu_{\lambda}(dx)\nu_{\lambda}(dy)\ge   
     2 \, \mathrm{Var}_{\nu_{\lambda}}(\phi).
\end{equation}
\end{theorem}

\begin{remark}
\begin{enumerate}
\item Equation \eqref{e:42} for $\lambda=0$ is the statement of Theorem \ref{t:6} for the case $V(x)=2x^{2}$ and  for $\lambda=1$ (more precisely, equation \eqref{e:91}) while  equation \eqref{e:43} is the statement of the second Poincar\'e inequality contained in Theorem \ref{t:7} for the semicircular law.  The combination of these two inequalities is equation \eqref{e:122}.

In other words, for measures $\nu_{\lambda}$,  the first Poincar\'e type inequality is driven by the comparison of the Jacobi and counting number operators defined in \eqref{e:jac} and \eqref{e:count}, as the second Poincar\'e type is the spectral gap of the counting number operator.  
\item Combining equations \eqref{e:42} and \eqref{e:43},  we also get a Brascamp-Lieb type inequality: 
\begin{equation}\label{e:400}
\int\phi'(x)^{2}(1-x^{2})\nu_{\lambda}(dx)\ge (2\lambda+1) \, \mathrm{Var}_{\nu_{\lambda}}(\phi).
\end{equation}
For $\lambda\ge 1/2$, the measure $\nu_{\lambda}$ is of the form $e^{-V(x)}dx$, where $V(x)=-c_{\lambda}-(\lambda-1/2)\log(1-x^{2})$, a strictly convex function on $(-1,1)$ and according to the classical Brascamp-Lieb
inequality \cite{Brascamp-Lieb},
\begin{equation}\label{e:401}
\int\phi'(x)^{2}\frac{(1-x^{2})^{2}}{(1+x^{2})} \, \nu_{\lambda}(dx)\ge (2\lambda-1) \, \mathrm{Var}_{\nu_{\lambda}}(\phi).
\end{equation}
Notice here that neither \eqref{e:400} not \eqref{e:401} implies the other which means that they complement each other in some sense.  For example if $\phi$ has support in $[-\frac{1}{2\lambda},\frac{1}{2\lambda}]$, \eqref{e:400} implies \eqref{e:401}, while if $\phi$ is supported on $[-1,1]\backslash [-\frac{1}{2\lambda},\frac{1}{2\lambda}]$, \eqref{e:401} implies \eqref{e:400}.  
\end{enumerate}
\end{remark}

\section{Wishart Ensembles and Marcenko-Pastur Distributions}

In this section, we address the preceding functional inequalities for probability measures on the real positive
axis in the context of the Wishart Ensembles from random matrix theory and their associated Marcenko-Pastur
distributions.

We start with the random matrix heuristics although, as far as we know, it has not been used towards functional
inequalities as before.
The problems of large deviations principle for the distribution of the eigenvalues of Wishart ensembles is discussed in  \cite{HP2}.  The model is as follows.  Take $T(n)$ a $n\times p(n)$ random matrix with all the entries being iid $N(0,1)$ random variables.  Then $T(n)T(n)^{t}$ for $n<p(n)$ is known as the nonsingular Wishart random ensemble.  According to \cite[page 129]{HP3}, the distribution of the Wishart ensembles is given by 
\[
C_{np} \, e^{-\frac{p(n)}{2}\mathrm{Tr} M}(\det M)^{(p-n-1)/2}dM.
\] 
where the measure $dM=\prod_{i\le j}dM_{ij}$ the restriction of the Lebegue measure on the set of $n\times n$ non-negative matrices. 

It is also known (for example \cite[page 129]{HP3}) that the joint distribution of eigenvalues $(\lambda_{1},\lambda_{2},\dots,\lambda_{n})$ of $\frac{1}{p(n)}T(n)T(n)^{t}$ is given by 
\[
\frac{1}{Z_{n}} \, e^{-\frac{p(n)}{2}\sum_{i=1}^{n} t_{i}}\prod_{i=1}^{n}\lambda_{i}^{(p(n)-n-1)/2}\prod_{1\le i < j\le n}|\lambda_{i}-\lambda_{j}|.
\]
Our interest is in the limit distribution of $\mu_{n}=\frac{1}{n}\sum_{i=1}^{n}\delta_{\lambda_{i}}$.  The classical result states that if $n/p(n)\xrightarrow[n\to\infty]{} \alpha\in (0,1]$, then the limit distribution of $\mu_{n}$ is the so called Marcenko-Pastur distribution given by 
\[
\mathbbm{1}_{[(1-\sqrt{\alpha})^{2},(1+\sqrt{\alpha})^{2}]}(x) \,
\frac{\sqrt{4\alpha-(x-1-\alpha)^{2}}}{2\pi \alpha x} \, dx.
\]
This is a particular model for the standard Wishart ensembles.  However one can consider a more general example with potentials for which the distribution of the matrix is driven by a potential $Q:[0,\infty)\to\R$,
\[
C_{n} \, e^{-p(n)\mathrm{Tr}Q(M)}(\det M)^{\gamma(n)}dM
\]
where $dM$ stands for the Lebesgue measure on $n\times n$ positive definite matrices.   The distribution of eigenvalues of $M$ is given by 
\[
\frac{1}{Z_{n}}e^{-p(n)\sum_{i=1}^{n} Q(t_{i})}\prod_{i=1}^{n}t_{i}^{\gamma(n)}\prod_{1\le i < j \le n}|t_{i}-t_{j}|.
\]  
The main result of \cite{HP2} is that the distribution of the random measures $\mu_{n}=\frac{1}{p(n)}\sum_{i=1}^{p(n)}\delta_{\lambda_{i}}$ under the conditions $n/p(n)\xrightarrow[n\to\infty]{}\alpha\in(0,1]$, $\gamma(n)/n\xrightarrow[n\to\infty]{}\gamma>0$, $\nu_{n}$ satisfy a large deviation principle with scale $n^{-2}$ and the rate function given by 
\[
R(\mu)= \tilde{E}_{Q}(\mu) - \inf_{\mu\in\mathcal{P}([0,\infty))} \tilde{E}_{Q}(\mu),
\]
where
\[
\tilde{E}_{Q}(\mu)=\int \alpha \big (Q(x)-\gamma\log (x)\big)\mu(dx)-\frac{\alpha^{2}}{2}
\iint\log|x-y|\mu(dx)\mu(dy).
\]

This gives the following motivation.  Assume that $V:[0,\infty)\to\R\cup\{+\infty\}$ is a lower semi-continuous potential such that  $\lim_{|x|\to\infty}(V(x)-2\log|x|)=\infty$.  Then,  according to the results in \cite{ST}, we know that there is a unique minimizer of 
\[
\inf_{\mu\in\mathcal{P}([0,\infty))} E_{V}(\mu).
\]
In addition the equilibrium measure $\mu_{V}$ has  compact support.    

A particular case of interest is $V(x)= r x - s \log (x)$ with $r>0,s\ge 0$ for which we know  \cite[page 207]{ST} that the equilibrium measure is given by 
\begin{equation}\label{e:121}
\mu_{V}(dx)=\mathbbm{1}_{[a,b]}(x) \, \frac{r\sqrt{(x-a)(b-x)}}{2 \pi x} \, dx\quad \text{where}\quad a=\frac{s+2-2\sqrt{s+1}}{r} \, ,  \quad b=\frac{s+2+2\sqrt{s+1}}{r} \, .
\end{equation}
One recovers the Marcenko-Pastur distribution for $V(x)= r x - s \log (x)$, $r >0$, $ s \geq 0$, with
$r=1/\alpha$ and $s=(1-\alpha)/\alpha$.  

The natural way to deal with functional inequalities in the context of measures on the positive axis $[0,\infty)$
is to transfer measures from $[0,\infty)$ into measures on the whole $\R$.  
For a measure $\mu$ on $[0,\infty)$, consider thus the associated symmetric measure $\tilde{\mu}$ on $\R$ defined as 
\begin{equation}\label{e:tilde}
\mu(F)=\tilde{\mu} \big (\{x:x^{2}\in F \} \big )
\end{equation}
for any measurable set $F$ of $[0,\infty)$.   Defining $\tilde{V}(x)=V(x^{2})/2$,  it is then an easy exercise to check that 
\begin{equation}\label{e:tilde3}
E_{V}(\mu)=2E_{\tilde{V}}(\tilde{\mu}).
\end{equation}
In addition, the minimizer of $E_{\tilde{V}}$ is $\mu_{\tilde{V}}=\tilde{\mu}_{V}$ 
Further, for the non-decreasing transportation map $\theta$ of $\mu_{V}$ into $\mu$, define 
\begin{equation}\label{e:tilde2}
\tilde{\theta}(x)=\sign(x)\sqrt{\theta(x^{2})},
\end{equation}
which transports $\tilde{\mu}_{\tilde{V}}$ into $\tilde{\mu}$.  

In addition, as it was pointed out in \cite{HPU1}, the relative free Fisher information $I_{V}(\mu)$ is defined for measures $\mu$ on $[0,\infty)$ with density $p=d\mu/dx$ in $L^{3}([0,\infty),xdx)$ as
\begin{equation}\label{e:tilde4}
I_{V}(\mu)=\int_{0}^{\infty} x \big (H\mu(x)-V'(x)\big )^{2}\mu(dx)
    \quad \text{with}\quad H\mu(x)=p.v. \int \frac{2}{x-y}\mu(dy).
\end{equation}
Otherwise we take $I_{V}(\mu)=+\infty$.  The main reason for defining this in this way is because,  cf. \cite[Lemma 6.3]{HPU1}  and the discussion following, one has
\begin{equation}\label{e:703}
I_{V}(\mu)=2I_{\tilde{V}}(\tilde{\mu}),
\end{equation}
where $I_{\tilde{V}}$ is defined by \eqref{e:27}.    

To state the transportation cost result, we define the appropriate distance.   For any $\mu,\nu\in\mathcal{P}([0,\infty))$, set the distance as
\begin{equation}\label{e:W2}
W(\mu,\nu)=\inf_{\pi\in\Pi(\mu,\nu)}\left( \int \big(\sqrt{x}-\sqrt{y}\big)^{2}\pi(dx,dy)\right)^{1/2}
\end{equation}
where $\Pi(\mu,\nu)$ is the set of probability measures on $\R^{2}$ with marginals $\mu$ and $\nu$.  

In this context we have the following transportation cost inequality.

\begin{theorem}\label{t:12}
Assume that $V:(0,\infty)\to \R$ is $C^{2}((0,\infty))$ such that  $V(x^{2})-\rho x^{2}$ is convex on $(0,\infty)$ for some $\rho>0$ and let $\mu_{V}$ be the equilibrium measure of $V$ on $[0,\infty)$.  Then, for any
probability measure $\mu$ on $[0,\infty)$, we have that 
\begin{equation}\label{e:80}
\rho \, W^{2}(\mu,\mu_{V})\le E_{V}(\mu)-E_{V}(\mu_{V}),
\end{equation}
In the case of $V(x)=rx-s\log(x)$ with $r>0$ and $s\ge0$, this inequality with $\rho=r$ is sharp.  
\end{theorem}

\begin{proof} As announced, the idea is to interpret this inequality as an inequality for potentials on the whole real line instead of $[0,\infty)$.  Using the measures $\tilde{\mu}$ and $\tilde{\mu}_{V}$ from equation \eqref{e:tilde}
together with \eqref{e:tilde3}, we have that 
\[
E_{V}(\mu)-E_{V}(\mu_{V})=2 \big (E_{\tilde{V}}(\tilde{\mu})-E_{\tilde{V}}(\tilde{\mu}_{V}) \big ).  
\]
On the other hand, if $\theta$ is the (increasing) transportation map of $\mu_{V}$ into $\mu$, then it is not hard to check that
\[
W^{2}(\mu,\nu)=\int \big (\sqrt{x}-\sqrt{\theta(x)} \big )^{2}\mu_{V}(dx)=\int \big (x-\tilde{\theta}(x) \big )^{2}\tilde{\mu}_{V}(dx).
\]
In this framework the inequality \eqref{e:80} translates as 
\begin{equation}\label{e:801}
\frac{\rho}{2} \, W_{2}^{2}(\tilde{\mu},\tilde{\mu}_{V})\le  E_{\tilde{V}}(\tilde{\mu})-E_{\tilde{V}}(\tilde{\mu}_{V}).
\end{equation}
From here we will use the same argument as in the proof of Theorem \ref{thm:1}.  Start with 
\begin{align*}
 E_{\tilde{V}}(\tilde{\mu})-E_{\tilde{V}}(\tilde{\mu}_{V}) 
 &=\int\left( \tilde{V}(\tilde{\theta}(x))-\tilde{V}(x)-\tilde{V}'(x)(\tilde{\theta}(x)-x)\right)\tilde{\mu}_{V}(dx)\\ 
&\quad+\iint \left(\frac{\tilde{\theta}(x)-\tilde{\theta}(y)}{x-y}-1- \log\frac{\tilde{\theta}(x)-\tilde{\theta}(y)}{x-y}\right)\tilde{\mu}_{V}(dx)\tilde{\mu}_{V}(dy). 
\end{align*}
and notice that the second line of this is non-negative.  For the first line we point out that because $\tilde{V}(x)-\frac{\rho}{2} x^{2}$  is convex and $x$ and $\tilde{\theta}(x)$ have the same sign, for any $x$,
\[
\tilde{V}(\tilde{\theta}(x))-\tilde{V}(x)-\tilde{V}'(x) \big (\tilde{\theta}(x)-x\big )
\ge \frac{\rho}{2}(\tilde{\theta}(x)-x)^{2},
\]
which implies \eqref{e:80}.  

In the case $V(x)=rx-s\log(x)$, take $\theta(x)=(\sqrt{x}+m)^{2}$ for large $m$ and notice that $\tilde{\theta}(x)=x+m\sign(x)$. Therefore inequality \eqref{e:801} becomes 
\[
rm^{2}\le rm^{2}+2r m \int |x|  \tilde{\mu}(dx) -2s\int \log\left(\frac{|x+m\sign(x)|}{|x|} \right)\tilde{\mu}(dx)-\iint \log\left( 1+m \, \frac{\sign(x)-\sign(y)}{x-y}\right)\tilde{\mu}(dx)\tilde{\mu}(dy)
\] 
which is sharp for large $m$.  \qedhere
\end{proof}

The next result is the Log-Sobolev type inequality, which was conjectured by Cabanal-Duvillard in \cite[page 140]{CD2} for the case of Marcenko-Pastur distribution.    

\begin{theorem}\label{t:13}  Let $V$ be as in the previous theorem.  Then, with the definition from \eqref{e:tilde4} and  for any measure $\mu\in\mathcal{P}([0,\infty))$, 
\begin{equation}\label{e:301}
E_{V}(\mu)-E_{V}(\mu_{V})\le \frac{1}{2\rho} \, I_{V}(\mu).
\end{equation}
In the case  $V(x)=rx-s\log (x) $, $r>0$ and $s\ge0$ inequality \eqref{e:301} with $\rho=r$ is sharp. 
\end{theorem}

\begin{proof}  We will discuss here the proof only in the case when  $\mu$ has a smooth compactly supported density, careful approximations being described in \cite{HPU1}.  

From \eqref{e:703}, we have $I_{V}(\mu)=2I_{\tilde{V}}(\tilde{\mu})$, where $I_{\tilde{V}}(\tilde{\mu})=\int (H\tilde{\mu}(x)-\tilde{V}'(x))^{2}\tilde{\mu}(dx)$.
Rewriting everything in terms of $\tilde{\mu}$ and the associated quantities, the inequality to be proven can be written in the same way as we did in the proof of Theorem \ref{t:3}, 
\begin{align}
\notag\frac{1}{2\rho}&\int \big (H\tilde{\mu}(\tilde{\theta}(x) \big ) -\tilde{V}'(\tilde{\theta}(x)))^{2}\tilde{\mu}_{\tilde{V}}(dx)+\int \left(\tilde{V}(x)-\tilde{V}(\tilde{\theta}(x))-\tilde{V}'(\tilde{\theta}(x))(x-\tilde{\theta}(x))\right)\tilde{\mu}_{\tilde{V}}(dx) \\ 
& \notag -\int \left(H\tilde{\mu}(\tilde{\theta}(x))-\tilde{V}'(\tilde{\theta}(x))\right) \big (x-\tilde{\theta}(x) \big )\tilde{\mu}_{\tilde{V}}(dx)  \\ 
& +\int H\tilde{\mu}(\tilde{\theta}(x)) \big (x-\tilde{\theta}(x) \big )\tilde{\mu}_{\tilde{V}}(dx) -\iint \log\frac{x-y}{\tilde{\theta}(x)-\tilde{\theta}(y)} \, \tilde{\mu}_{\tilde{V}}(dx)\tilde{\mu}_{\tilde{V}}(dy)\ge0. \label{e:228}
\end{align}
Notice that $\tilde{V}(x)-\frac{\rho}{2} x^{2}$ is not convex on the whole real line but it is convex on the intervals $(0,\infty)$ and $(-\infty,0)$.   The key to everything here is that $\tilde{\theta}(x)$ has the same sign as $x$ and this allows us to apply convexity of $\tilde{V}(x)-\frac{\rho}{2} x^{2}$ on each of the intervals $(-\infty,0)$ and $(0,\infty)$ to conclude that 
\begin{equation}\label{e:223}
\tilde{V}(x)-\tilde{V} \big (\tilde{\theta}(x) \big )-\tilde{V}'(\tilde{\theta}(x))(x-\tilde{\theta}(x))
\ge\frac{\rho}{2} \big (x^{2}-\tilde{\theta}(x)^{2}-2\tilde{\theta}(x)(x-\tilde{\theta}(x)) \big)
    =\frac{\rho}{2} \big (x-\tilde{\theta}(x) \big )^{2}.
\end{equation}
From here we can follow word by word the proof of Theorem \ref{t:3}.  

For the case $V(x)=rx$, we have equality in \eqref{e:301} if $\tilde{\theta}(x)=x+m\sign(x)$ and thus this means $\theta(x)=(\sqrt{x}+m)^{2}$.  

In the case $V(x)=rx-s\log (x) $, we look at $\tilde{\theta}(x)=x+m$ for large $m$.  In this case $\tilde{V}(x)=rx^{2}/2-s\log |x|$ and then a simple calculation shows that \eqref{e:301} is equivalent to 
\begin{align*}
rm^{2}+&2mr\int |x|\tilde{\mu}_{V}(dx)-2s\int \log\left(\frac{|x+m\sign(x)|}{|x|}\right)\tilde{\mu}_{V}(dx) -2\iint \log\left(1+m\frac{\sign(x)-\sign(y)}{x-y}\right)\tilde{\mu}(dx)\tilde{\mu}(dy) \\ &\le \frac{m^{2}}{\rho}\int\left(r-\frac{s}{x(x+m\sign(x))}\right)^{2}\tilde{\mu}_{V}(dx) .
\end{align*}
Dividing both sides by $m^{2}$ and taking the limit of $m$ to infinity implies that $\rho\le r$.  On the other hand $\rho=r$ validates \eqref{e:301}, hence $\rho=r$ is the best constant.  \qedhere
\end{proof}

Next in line is the HWI inequality which is the content of the following statement.  

\begin{theorem}
Assume $V$ is as in Theorem \ref{t:12} and the distance $W$ given by \eqref{e:W2}.   Then for any measure $\mu\in\mathcal{P}([0,\infty))$,  
\begin{equation}\label{hwi-w}
E_{V}(\mu)-E_{V}(\mu_{V})\le  \, \sqrt{2I_{V}(\mu)}W(\mu,\mu_{V})- \rho \, W^{2}(\mu,\mu_{V}).
\end{equation}
For the case of $V(x)=rx-s\log (x) $, $r>0$, $s\ge0$, this inequality for $\rho=r$ is sharp.  
\end{theorem}

\begin{proof}  As it was made clear in the previous two theorems, we translate this inequality in terms of the associated symmetric measures on $\R$.  Following upon the proofs of above theorems, we can rewrite \eqref{hwi-w} in the following form:  
\begin{align*}
\bigg(\int (H\tilde{\mu}(\tilde{\theta}(x))-\tilde{V}'(\tilde{\theta}(x)))^{2}&\tilde{\mu}_{V}(dx) \int (\tilde{\theta}(x)-x)^{2}\tilde{\mu}_{V}(dx)\bigg)^{1/2}-\int \left(H\tilde{\mu}(\tilde{\theta}(x))-\tilde{V}'(\tilde{\theta}(x))\right)(x-\tilde{\theta}(x))\tilde{\mu}_{V}(dx)  \\ 
&+\int \left(\tilde{V}(x)-\tilde{V}(\tilde{\theta}(x))-\tilde{V}'(\tilde{\theta}(x)) \big (x-\tilde{\theta}(x)\big )-\rho(\tilde{\theta}(x)-x)^{2}\right)\tilde{\mu}_{V}(dx) \\ 
&+\int H\tilde{\mu}(\tilde{\theta}(x)) \big (x-\tilde{\theta}(x) \big )\tilde{\mu}_{V}(dx) -\iint \log\frac{x-y}{\tilde{\theta}(x)-\tilde{\theta}(y)}\tilde{\mu}_{V}(dx)\tilde{\mu}_{V}(dy) \ge0 .
\end{align*}
Using the fact that $\tilde{V}(x)-\frac{\rho}{2} x^{2}$ is convex on each interval $(-\infty,0)$ and $(0,\infty)$ combined with the fact that $x$ and $\tilde{\theta}(x)$ have the same sign, the rest of the proof is the same as the one of Theorem \ref{t:hwi}.  

For the case $V(x)=rx-s\log (x) $, using $\theta(x)=(\sqrt{x}+m)^{2}$, one can show that $\rho=r$ is sharp. \qedhere
\end{proof}

At last, we would like to discuss a Poincar\'e type inequality in this context. As in Section \ref{s:p1}, for the heuristics, we consider the general model of random matrices with distribution
\begin{equation}\label{e:120}
\p_{n}(dM)=C_{n} \, e^{-nr\mathrm{Tr}M }(\det M)^{sn}dM=C_{n} \, e^{-n\mathrm{Tr} \big (rM-s\log (M) \big )}dM
 =C_{n} \, e^{-n\mathrm{Tr}(V(M))}dM
\end{equation}
where  $dM$ stands for the Lebesgue measure on $n\times n$ positive definite matrices and $s\ge0$.  For a given  smooth compactly supported function $\phi:[0,\infty)\to\R$, we want to apply the Brascamp-Lieb inequality \cite{Brascamp-Lieb} to the function $\Phi(M)=\mathrm{Tr}\phi(M)$ on the space of positive definite matrices.  Now,  $\nabla\Phi(M)=\phi'(M)$.   

The Hessian of $\Psi(M):=\mathrm{Tr}(V(M))$ can be interpreted as a linear map from $\mathcal{H}_{n}$ ($n\times n$ Hermitian matrices) into itself which is given by $\nabla^{2}\Psi(M)X = sM^{-1}XM^{-1}$.  Hence the inverse of the Hessian is then $(\nabla^{2}\Psi(M))^{-1}X=\frac{1}{s}MXM$.  Thus we obtain from Brascamp-Lieb that
\[
\int \frac{1}{n} \, \mathrm{Tr} \big ((\nabla^{2}\Psi(M))^{-1}\phi'(M)^{2} \big)\p_{n}(dM)
   \ge \var_{\p_{n}} \big (\Phi(M) \big ).
\]
On the other hand, from \cite{Jon} or \cite{CD1} the variance of $\Phi(M)$ converges to $\frac{1}{4}\var_{\mathrm{arcsine}_{[a,b]}}(\phi )$, where we recall that $\mathrm{arcsine}_{[a,b]}=\frac{dx}{\pi\sqrt{(x-a)(b-x)}}$ is the arcsine law on the support $[a,b]$ of $\mu_{V}$.   
Next, $\frac{1}{n}\mathrm{Tr}((\nabla^{2}\Psi(M))^{-1}\phi'(M)^{2})=\frac{1}{s n}\mathrm{Tr}((\phi'(M)M)^{2})$, whose integral against $\p_{n}$ converges to the integral of $\frac{1}{s}x^{2}\phi'(x)^{2}$ against the equilibrium measure $\mu_{V}$  from equation \eqref{e:121}.  These considerations suggest that 
\begin{equation}\label{e:130}
\int x^{2}\phi'(x)^{2}\mu_{V}(dx)\ge \frac{s}{4} \, \var_{\mathrm{arcsine}_{[a,b]}}(\phi).
\end{equation}
Notice here that one can actually make this heuristic into an actual proof of this inequality.  

Motivated by these heuristics and also inspired by Theorem \ref{t:6}, we have the following stronger result.

\begin{theorem}
Assume that $Q:[0,\infty)\to\R$ is a convex potential and let $V(x)=Q(x)-s\log(x)$ for $s>0$ satisfy $\lim_{x\to\infty}(V(x)-2\log (x))=\infty$.   Assume that the support of $\mu_{V}$ is $[a,b]$.  Then  for any smooth function $\phi$ on $[a,b]$, the following holds, 
\begin{equation}\label{e:203}
\int x^{2}\phi'(x)^{2}\mu_{V}(dx)\ge \frac{s}{4\pi^{2}}\int_{a}^{b}\int_{a}^{b}\left(\frac{\phi(x)-\phi(y)}{x-y}\right)^{2}\frac{-2ab+(a+b)(x+y)-2xy}{2\sqrt{(x-a)(b-x)}\sqrt{(y-a)(b-y)}} \, dxdy.
\end{equation}
If $Q(x)=rx+t$, equality is attained for $\phi(x)=c_{1}+\frac{c_{2}}{x}$, therefore \eqref{e:203} is sharp.  

In particular, combining \eqref{e:203} with \eqref{e:43} for $\lambda=0$, we get an improvement of  \eqref{e:130} as
\[
\int x^{2}\phi'(x)^{2}\mu_{V}(dx)\ge \frac{s}{2} \, \var_{\mathrm{arcsine}_{[a,b]}}(\phi).
\]
Equality though  is attained only for $\phi$ identically $0$.  

In the case $V(x)=rx$, $r>0$, on $[0,\infty)$,  there is no constant $C>0$ such that inequality
\eqref{e:203} holds with $C$ instead of $s/4\pi^2$. Nevertheless, for every smooth $\phi$ on $[a,b]$, the following holds,
\begin{equation}\label{e:321}
\int x\phi'(x)^{2}\mu_{V}(dx)\ge \frac{r}{4\pi^{2}}\int_{a}^{b}\int_{a}^{b}\left(\frac{\phi(x)-\phi(y)}{x-y}\right)^{2}\frac{-2ab+(a+b)(x+y)-2xy}{2\sqrt{(x-a)(b-x)}\sqrt{(y-a)(b-y)}} \, dxdy,
\end{equation}
with equality for $\phi(x)=c_{1}+c_{2}x$.
\end{theorem}

As remarked after the statement of Theorem~\ref{t:6}, the numerator in \eqref{e:321} is nonnegative.

\begin{proof}   The same argument as in the proof of Theorem \ref{t:6}, shows that the density $g(x)$ of $\mu_{V}$ satisfies
 \[
g(x)\ge\frac{ s\sqrt{(x-a)(b-x)}}{2\pi x\sqrt{ab}} \, ,
 \]
 therefore it suffices to show that 
 \[
 \frac{1}{\pi\sqrt{ab}}\int_{a}^{b} x\phi'(x)^{2}\sqrt{(x-a)(b-x)} \, dx\ge \frac{1}{2\pi^{2}}\int_{a}^{b}\int_{a}^{b}\left(\frac{\phi(x)-\phi(y)}{x-y}\right)^{2}\frac{-2ab+(a+b)(x+y)-2xy}{2\sqrt{(x-a)(b-x)}\sqrt{(y-a)(b-y)}} \, dxdy.
 \]
 Next, making the change of variable $x=(a+b)/2+u(b-a)/2$ and denoting $\zeta(u)=\phi((a+b)/2+u(b-a)/2)$, we reduce the problem to showing that  for any  smooth function $\phi$ on $[-1,1]$, we have 
 \[
  \frac{1}{\pi\sqrt{ab}}\int_{-1}^{1} \left(\frac{a+b}{2}+\frac{b-a}{2}u\right)\zeta'(u)^{2}\sqrt{1-u^{2}} \, du\ge  \frac{1}{2\pi^{2}}\int_{-1}^{1}\int_{-1}^{1}\left(\frac{\zeta(u)-\zeta(v)}{u-v}\right)^{2}\frac{1-uv}{\sqrt{1-u^{2}}\sqrt{1-v^{2}}} \, dudv.
 \]
Denoting $\beta=\frac{b-a}{b+a}$, we have that $\frac{a+b}{2\sqrt{ab}}=\frac{1}{\sqrt{1-\beta^{2}}}$, and the
preceding inequality reformulates as
 \begin{equation}\label{e:131}
 \int (1+\beta u)\zeta'(u)^{2}\sqrt{1-u^{2}} \, du
 \ge \frac{\sqrt{1-\beta^{2}}}{2\pi}\int_{-1}^{1}\int_{-1}^{1}\left(\frac{\zeta(u)-\zeta(v)}{u-v}\right)^{2}\frac{1-uv}{\sqrt{1-u^{2}}\sqrt{1-v^{2}}} \, dudv.
 \end{equation}
 To show this, take $\psi(t)=\zeta(\cos(t))$ and then after the change of variable $u=\cos(t)$ we need to check
 \[
 \int_{0}^{\pi} (1+\beta \cos(t))\psi'(t)^{2}dt\ge \frac{\sqrt{1-\beta^{2}}}{2\pi}\int_{0}^{\pi}\int_{0}^{\pi}\left(\frac{\psi(t)-\psi(s)}{\cos(t)-\cos(s)} \right)^{2}(1-\cos(t)\cos(s))dtds.
 \]
 Writing $\psi(t)=\sum_{n=0}^{\infty}a_{n}\cos(nt)$ and using that $\psi'(t)=-\sum_{n=1}^{\infty}n a_{n}\sin(nt)$, together with the fact that 
 \[
 \int_{0}^{\pi}\cos(t)\sin(n t)\sin(mt)dt=\begin{cases}
\frac{\pi}{4} & \text{for}\: |m-n|=1 \\ 
0 & \text{otherwise},
 \end{cases}
 \] 
 and equation \eqref{e:110},  the inequality becomes 
 \begin{equation}\label{e:331}
 \sum_{n\ge1}( n^{2}a_{n}^{2}+\beta n(n+1)a_{n}a_{n+1} ) \ge \sqrt{1-\beta^{2}}\sum_{n\ge1}na_{n}^{2}.
 \end{equation}
 Let $\delta=\frac{1-\sqrt{1-\beta^{2}}}{\beta}$ be the solution $0<\delta <1$ of $\beta\delta^{2}-2\delta+\beta=0$.    Notice that for any $n\ge1$, we have
 \[
 a_{n}a_{n+1}\ge -\frac{\delta}{2} \, a_{n}^{2}-\frac{1}{2\delta} \, a_{n+1}^{2}
 \]
 which implies that 
 \[
 \sum_{n\ge1} \big ( n^{2}a_{n}^{2}+\beta n(n+1)a_{n}a_{n+1} \big)
 \ge  \sum_{n\ge1}\left( n^{2}a_{n}^{2}-\frac{\beta n(n+1)}{2}\left(\delta a_{n}^{2}+\frac{1}{\delta}a_{n+1}^{2}\right) \right) =\sum_{n\ge1}\frac{n\beta(1-\delta^{2})}{2\delta} \, a_{n}^{2}=\sqrt{1-\beta^{2}}\sum_{n\ge1}na_{n}^{2},
 \]
 what we had to prove.  Notice here that equality is attained in this inequality if and only if $a_{n+1}=-\delta a_{n}$ for all $n\ge1$, which means that $a_{n}=(-1)^{n-1}\delta^{n-1}a_{1}$.  This corresponds to the function $\psi(t)=a_{1}\frac{\delta+\cos t}{1+\delta^{2}+2\delta \cos t}$, or $\zeta(u)=a_{1}\frac{\delta+u}{1+\delta^{2}+2\delta u}$ which means that $\phi(x)=a_{1}(r-s/x)$.  Therefore equality holds also for $\phi(x)=c_{1}+c_{2}/x$.
 
 For the second part, in the case $V(x)=r x$ with $r>0$,  notice that if there is a $C>0$ so that
 \eqref{e:203} holds with $C$ instead of $s/4\pi^2$, then, following the same argument as above, we would have the equivalent of \eqref{e:331} as 
 \[
 \sum_{n\ge1}\big ( n^{2}a_{n}^{2}+ n(n+1)a_{n}a_{n+1} \big ) \ge C\sum_{n\ge1}na_{n}^{2}.
 \]
 Taking in this $a_{n}=\frac{(-\gamma)^{n}}{n}$ for $0<\gamma<1$, we have that 
$ \gamma^2 / (\gamma +1) \ge -C\log(1-\gamma^{2})$,
 and this is certainly false for $\gamma$ close to $1$.  
 
 For equation \eqref{e:321}, notice that in this case the equilibrium measure is $\mu_{V}(dx)=\frac{r\sqrt{b-x}}{2\pi\sqrt{x}}$ and then after a simple rescaling this follows from equation \eqref{e:91}.
 This complete the proof of the theorem.  \qedhere
 
\end{proof}

It is interesting to look at this inequality as a spectral gap result as in Section \ref{s:spectral}.  For example in the case of the Marcenko-Pastur measure ($Q(x)=rx$), the inequality \eqref{e:203} is actually equivalent to inequality \eqref{e:131}.   Using the interpretation from Section \ref{s:spectral}, we can rephrase this as, for a given $\beta\in (0,1)$,  
\[
\int (1+\beta x)(1-x^{2})\phi'(x)^{2}\nu_{0}(dx)\ge \sqrt{1-\beta^{2}} \, \langle N\phi,\phi \rangle_{\nu_{0}}
\]
where $\nu_{0}$ is the arcsine law on $[-1,1]$ and  $N$ is the number
 operator.  Now we can define the operator 
\[
L_{\beta}\phi(x)=-(1+\beta x)(1-x^{2})\phi''(x)-(\beta-x-2\beta x^{2})\phi'(x).
\]
With this definition, 
\[
\langle L_{\beta}\phi,\phi \rangle_{\nu_{0}}=\frac{1}{\pi}\int (1+\beta x)\phi'(x)^{2}\sqrt{1-x^{2}} \, dx
\]
and then inequality \eqref{e:131} becomes 
\[
\langle L_{\beta}\phi,\phi \rangle_{\nu_{0}}\ge \sqrt{1-\beta^{2}} \, \langle N\phi,\phi \rangle_{\nu_{0}}
\]
for any smooth function $\phi $ on $[-1,1]$.  In particular this means that
$L_{\beta}\ge \sqrt{1-\beta^{2}}N$.
On the other hand it is clear that the operator $L_{\beta}$ can not be diagonalized by the Chebyshev polynomials of the first kind, therefore the orthogonal polynomial approach given in Section \ref{s:spectral} does not work the same way here.  

\bigskip

\begin{remark}
We want to point out that for the case $V(x)=rx-s\log(x)$ for $r>0$  and $s\ge0$, the parameter $r$ appears in the transportation, Log-Sobolev and HWI, while the parameter $s$ plays the dominant role in the Poincar\'e inequality.   
\end{remark}

 {\bf Acknowledgements. }
We would like to thank D. Cabanal-Duvillard for pointing to us the formula of the fluctuation for Wishart ensembles
and for informing us about his Log-Sobolev conjecture in \cite {CD2}.  Many thanks to the anonymous referee for the pertinent and scholarly comments which pointed several shortcomings of the submitted version and led to an overall improvement of this paper.

\end{document}